\documentclass[12pt]{amsart}
\usepackage{amssymb, amsmath}

\usepackage[all]{xy}

\numberwithin{equation}{section}

\newcommand{\mm}{\mathfrak m}
\newcommand{\N}{\mathbb{N}}
\newcommand{\PP}{\mathbb{P}}

\newcommand{\Z}{\mathbb{Z}}
\newcommand{\Ic}{\mathcal{I}}

\newcommand{\Mcc}{\mathcal{M}}
\newcommand{\Nc}{\mathcal{N}}

\newcommand{\Ab}{{\bf A}}

\newcommand{\Bb}{{\bf B}}

\newcommand{\Fb}{{\bf F}}

\newcommand{\Ib}{{\bf I}}
\newcommand{\Fi}[1]{{\bf F}^{\FI, #1}}
\newcommand{\Fo}[1]{{\bf F}^{\FIO, #1}}
\newcommand{\Kb}{{\bf K}}
\newcommand{\M}{{\bf M}}
\newcommand{\Nb}{{\bf N}}
\newcommand{\Pb}{{\bf P}}

\newcommand{\Xb}{{\bf X}}
\newcommand{\XI}[1]{{\bf X}^{\FI, #1}}
\newcommand{\XO}[1]{{\bf X}^{\FIO, #1}}

\DeclareMathOperator{\Alg}{Alg}

\DeclareMathOperator{\FI}{FI}

\DeclareMathOperator{\Hom}{Hom}
\DeclareMathOperator{\id}{id}
\DeclareMathOperator{\im}{im}
\DeclareMathOperator{\ini}{in}
\DeclareMathOperator{\Inc}{Inc (\N)}
\DeclareMathOperator{\ind}{ind}

\DeclareMathOperator{\lc}{lc}

\DeclareMathOperator{\lm}{lm}
\DeclareMathOperator{\lt}{lt}
\DeclareMathOperator{\Mod}{Mod}

\DeclareMathOperator{\FIO}{OI}

\DeclareMathOperator{\reg}{reg}

\DeclareMathOperator{\sgn}{sgn}

\DeclareMathOperator{\Sym}{Sym}
\DeclareMathOperator{\tensor}{\otimes}
\DeclareMathOperator{\Tor}{Tor}

\DeclareMathOperator{\pnt}{\raise 0.5mm \hbox{\large\bf.}}
\DeclareMathOperator{\lpnt}{\hbox{\large\bf.}}

\newcommand{\la}{\langle}
\newcommand{\ra}{\rangle}

\newcommand{\s}{\; | \;}

\let\phi=\varphi
\newcommand{\eps}{\varepsilon}
\newtheorem{thm}{\bf Theorem}[section]
\newtheorem{lem}[thm]{\bf Lemma}
\newtheorem{cor}[thm]{\bf Corollary}
\newtheorem{prop}[thm]{\bf Proposition}
\newtheorem{conj}[thm]{\bf Conjecture}

\theoremstyle{definition}
\newtheorem{defn}[thm]{\bf Definition}

\newtheorem{rem}[thm]{\bf Remark}
\newtheorem{ex}[thm]{\bf Example}

\title{$\FI$- and $\FIO$-modules with varying coefficients}

\author[Uwe Nagel]{Uwe Nagel}
\address{Department of Mathematics, University of Kentucky, 715 Patterson Office Tower, Lexington, KY 40506-0027, USA}
\email{uwe.nagel@uky.edu}

\author{Tim R\"omer}
\address{Universit\"at Osnabr\"uck, Institut f\"ur Mathematik, 49069 Osnabr\"uck, Germany}
\email{troemer@uos.de}

\begin{document}

\begin{abstract}
We introduce $\FI$-algebras over a commutative ring $K$ and the category of $\FI$-modules over an $\FI$-algebra.
Such a module may be considered as a
family of invariant modules over compatible varying $K$-algebras. $\FI$-modules over $K$ correspond to the well studied constant coefficient case where every algebra equals $K$.
We show that a finitely generated $\FI$-module over a noetherian polynomial $\FI$-algebra is a noetherian module. This is established by introducing $\FIO$-modules. We prove that every submodule of a finitely generated free $\FIO$-module over a noetherian polynomial $\FIO$-algebra has a finite Gr\"obner basis. Applying our noetherianity results to a family of free resolutions, finite generation translates into stabilization of syzygies in any fixed homological degree. In particular, in the graded case
this gives uniformity results on degrees of minimal syzygies.
\end{abstract}

\thanks{The first author was partially supported by Simons Foundation grant \#317096.
He also is grateful to Daniel Erman and Greg Smith for organizing a very inspiring April 2016 Banff workshop on Free Resolutions, Representations, and Asymptotic Algebra.}

\maketitle



\section{Introduction}

Denote by $\FI$ the category whose objects are finite sets and whose morphisms are injections. An $\FI$-module over a commutative
ring $K$ with unity is a functor from $\FI$ to the category of $K$-modules.
$\FI$-modules over $K$ and its relatives provide a framework for studying a sequence of representations of symmetric and related groups
on finite-dimensional vector spaces of varying dimensions (see, e.g., \cite{CE, CEF, CEFN, PS, SS-14, SS-16}). Stabilization results are
derived as a consequence of finite generation of a suitable
$\FI$-module.
A cornerstone of the theory is that, if $K$ is noetherian, then a finitely generated $\FI$-module over $K$ is noetherian, that is, all its $\FI$-submodules are also finitely generated.
In a parallel development, originally motivated by questions in algebraic statistics, it was shown that
every ideal in a polynomial ring $K[X]$ in infinitely many variables $x_1, x_2,\ldots$ over a field $K$ that is invariant under the action of the symmetric group $\Sym (\N)$ is generated by finitely many $\Sym (\N)$-orbits (see \cite{AH, C, Draisma, HS} and also \cite{Draisma-factor, DK, SS-12b} for related results).

In this paper we introduce and utilize an extension of both
research strands by studying families of modules over varying rings.

More precisely, an $\FI$-algebra over a commutative ring $K$ is a functor $\Ab$ from $\FI$ to the category
of
commutative, associative, unital $K$-algebras with $\Ab (\emptyset) = K$. For example, we denote
by $\Xb$ the $\FI$-algebra with $\Xb_n = \Xb ([n]) = K[x_1,\ldots,x_n]$ and homomorphisms $\Xb (\eps)$
determined by $\eps (x_i) = x_{\eps (i)}$, where $\eps\colon [m] \to [n] = \{1,2,\ldots,n\}$ is
any morphism. Note that $K$ itself may be considered as the \emph{constant} $\FI$-algebra that maps
every finite set onto $K$. An $\FI$-module over an $\FI$-algebra $\Ab$ is a covariant functor $\M$ from
$\FI$ to the category of $K$-modules which, informally, is compatible with the $\FI$-algebra
structure of $\Ab$ (see Definition~\ref{def:OI-module} for details). It is finitely generated if there is a
finite subset that is not contained in any proper $\FI$-submodule. As in the classical case, $\M$ is
said to be noetherian if every $\FI$-submodule of $\M$ is finitely generated.

Given an $\FI$-module $\M$ over an $\FI$-algebra $\Ab$, there are natural colimits $\lim \M$ and
$K[\Ab]$ such that $\lim \M$ is a $K[\Ab]$-module. For example, if $\Ib$ is an ideal of $\Xb$, then
$\lim \Ib$ is a $\Sym (\N)$-invariant ideal of $K[X]$.

If $\M$ is noetherian, we show that $\lim \M$ is
$\Sym (\N)$-noetherian, that is, every $\Sym (\N)$-invariant submodule of $\lim \M$ is generated by
finitely many $\Sym (\N)$-orbits (see Theorem~\ref{thm:noeth-and-limits}). Since, as a special case of our results, $\Xb$ is a noetherian $\FI$-algebra over $K$, this implies in particular \cite[Theorems 1.1]{HS}, that is, $\Sym (\N)$-invariant ideals of $K[X]$ are generated by finitely many $\Sym (\N)$-orbits (see Corollary \ref{cor:limits-alg-noeth}).

In contrast to the category of $\FI$-modules over $K$, it is not true that every finitely generated
$\FI$-module over any $\FI$-algebra is noetherian. However, we prove that finitely generated
$\FI$-modules over (any finite tensor power of) $\Xb$ are noetherian
(see Theorem~\ref{thm:fg-gives-noeth-mod} and Corollary~\ref{cor:noeth-subalgebra} for a more
general result). Since the constant $\FI$-algebra $K$ is a quotient of $\Xb$, this also covers the
noetherianity result for $\FI$-modules over $K$ mentioned above.

As an application, we establish stabilization of $p$-syzygies for fixed $p$. More precisely, our
results yield that every finitely generated $\FI$-module $\M$ over $\Xb$ admits a free resolution
over $\Xb$, where each free module is again finitely generated. In particular, for every $p \in \N$, the
$p$-th syzygy module of $\M$ is a finitely generated $\FI$-module. This implies that there are finitely many \emph{master syzygies} whose orbits generate the $p$-th syzygy module of
every module $\M ([n])$ over $\Xb_n$ if $n$ is sufficiently large. In particular, for graded modules this shows that there
is a uniform upper bound for the degrees of minimal $p$-th syzygies of every module $\M ([n])$. Uniformity means that the bound is independent of $n$. However, examples show that this bound depends on the homological degree $p$ in general. Results of this kind appeared in \cite{CE, S-16, S-17, SS-14, S}; see the discussion in Remark \ref{rem:stabilizationetc}.

In order to establish the above results, we also introduce $\FIO$-modules over an $\FIO$-algebra. They are
defined analogously to their $\FI$-counterparts as functors from the category $\FIO$. The objects of $\FIO$
are totally ordered finite sets and its morphisms are order-preserving injective maps. Every $\FI$-module may be considered as an $\FIO$-module. The key technical advantage of $\FIO$-modules is that they are amenable to a theory of Gr\"obner bases. In fact, every finitely generated free $\FIO$-module over $\Xb$ has a finite Gr\"obner basis (see Theorem~\ref{thm:finite-G-basis}).

This paper is organized as follows. $\FIO$- and $\FI$-algebras are introduced in Section \ref{sec:fio-algebras}. In particular, we define polynomial algebras in this context.
Every polynomial algebra is generated by a single element, and an algebra is finitely generated if and only if it is a quotient of a finite tensor product of polynomial algebras (see Proposition~\ref{prop:fg-FOI-algebra}).

In Section~\ref{sec:fio-modules}, we formally define $\FI$- and $\FIO$-modules over a corresponding algebra $\Ab$. Over any $\FI$- or $\FIO$-algebra $\Ab$, we introduce a class of modules that are the building blocks of the free modules (see Definition~\ref{def:free}). A module over $\Ab$ is finitely generated if and only if it is a quotient of a finitely generated free module (see Proposition~\ref{prop:finite-gen}).

General properties of noetherian $\FI$-algebras and $\FI$-modules are discussed in Section~\ref{sec:noeth-gen}. In particular, we show there that a noetherian module has a finitely generated colimit.

Section~\ref{sec:orders} is devoted to a combinatorial topic. The goal is to establish that a certain partial order is a well-partial-order (see Proposition~\ref{prop:Higman-on-products}). This is a key ingredient of the central result in Section~\ref{sec:groebner}: Every finitely generated free $\FIO$-module over the polynomial $\FIO$-algebra analog of $\Xb$ has a finite Gr\"obner basis with respect to any monomial order (see Theorem~\ref{thm:finite-G-basis}). We use it to derive the mentioned results about noetherian modules. Moreover, we show that certain subalgebras of $\Xb$ and its $\FIO$-analog such as Veronese subalgebras are again noetherian and that finitely generated modules over these noetherian algebras are noetherian modules (see Proposition~\ref{prop:noeth-subalgebra}).

In Section~\ref{sec:syzygies}, we discuss free resolutions of $\FI$ and $\FIO$-modules. The mentioned stabilization results of $p$-syzygies are established in Theorems \ref{thm:stabilization-syz} and \ref{thm:stabilization}.

Finally, in Section~\ref{sec:koszul} we present an analog of the classical Koszul complex for
 $\FIO$-modules (see Proposition~\ref{prop:Koszul-exact}). It gives specific examples of resolutions that are used to establish the stabilization of syzygies in Section~\ref{sec:syzygies}.


\section{$\FI$- and $\FIO$-algebras}
\label{sec:fio-algebras}

Before introducing modules we define the algebras from which the coefficients will be drawn. Our exposition is influenced by the approach in \cite{CEF}.

\begin{defn}
Denote by $\FI$ the category whose objects are finite sets and whose morphisms are injections (see \cite{CEF} for more details).

The category $\FIO$ is the subcategory of $\FI$ whose objects are totally ordered finite sets and whose morphisms are order-preserving injective maps (see \cite{SS-14}).
\end{defn}

For an integer $n \ge 0$, we set $[n] = \{1,2,\ldots,n\}$. Thus, $[0] = \emptyset$. We denote by $\N$ and $\N_0$ the set of positive integers and non-negative integers, respectively.

\begin{rem}
The category $\FIO$ is equivalent to the category with objects $[n]$ for $n \in \N_0$ and morphisms being order-preserving injective maps $\eps\colon [m] \to [n]$. In particular, this implies $\eps (m) \ge m$.
\end{rem}

For later use we record the following observation. Its proof uses maps
\begin{align}
\label{eq:iota}
\iota_{m, n}\colon [m] \to [n], \ j \mapsto j,
\end{align}
where $m \le n$ are any positive integers. They are $\FIO$ morphisms.

\begin{lem}
\label{lem:decompose}
For any positive integers $m < n$ one has
\[
\Hom_{\FIO} ([m], [n]) = \Hom_{\FIO} ([m+1], [n]) \circ \Hom_{\FIO} ([m], [m+1])
\]
and
\[
\Hom_{\FI} ([m], [n]) = \Hom_{\FI} ([m+1], [n]) \circ \Hom_{\FI} ([m], [m+1]).
\]
\end{lem}

\begin{proof}
In both cases the right-hand side is obviously contained in the left-hand side. The reverse inclusion in the first case follows by the arguments for \cite[Proposition 4.6]{NR} by replacing $\Inc_{m, n}$ by $\Hom_{\FIO} ([m], [n])$. In the second case, consider any $\eps \in \Hom_{\FI} ([m], [n])$, and choose some $i \in [n] \setminus \im \eps$. Define $\widetilde{\eps}\colon [m+1] \to [n]$ by
\[
\widetilde{\eps} (j) = \begin{cases}
\eps (j) & \text{ if } j \in [m],\\
i & \text{ if } j = m+1.
\end{cases}
\]
Then we get $\eps = \widetilde{\eps} \circ \iota_{m, m+1}$, which concludes the proof.
\end{proof}

Let $K$ be a commutative ring with unity.
Denote by $K$-$\Alg$ the category of commutative, associative, unital $K$-algebras whose morphisms are $K$-algebra homomorphisms that map the identity of the domain onto the identity of the codomain.

\begin{defn}
\label{def:OI-algebra}
\
\begin{enumerate}
\item
An \emph{$\FIO$-algebra over $K$}
 is a covariant functor $\Ab$ from $\FIO$ to the category $K$-$\Alg$ with $\Ab({\emptyset}) = K$.
\item
An \emph{$\FI$-algebra over $K$} is defined analogously as a functor $\Ab$ from the category $\FI$ to $K$-$\Alg$ with $\Ab({\emptyset}) = K$.
\end{enumerate}
\end{defn}

Since $\FIO$ is a subcategory of $\FI$, any $\FI$-algebra may also be considered as an $\FIO$-algebra. We often will use the same symbol to denote both of these algebras.

For a finite set $S$, we write $\Ab_S$ for the $K$-algebra $\Ab(S)$, and we denote $\Ab_{[n]}$ by $\Ab_n$. Given a morphism $\eps\colon S \to T$, we often write $\eps^*\colon \Ab_S \to \Ab_T$ for the morphism $\Ab(\eps)$.

\begin{ex}
\label{exa:OI-alg}
\
\begin{enumerate}
\item

Fix an integer $c > 0$, and consider a polynomial ring
\[
K[X] := K[x_{i, j} \; | \; i \in [c],\ j \in \N].
\]
It naturally gives rise to an $\FI$-algebra as well as an $\FIO$-algebra $\Pb$ as follows: For a subset $S \subset \N$, set $\Pb_S = K[x_{i, j} \; | \; i \in [c],\ j \in S]$.
Given an $\FI$ or $\FIO$ morphism $\eps\colon S \to T$ of subsets of $\N$, define
$$
\eps^*\colon \Pb_S \to \Pb_T \text{ by }
\eps^* (x_{i, j}) = x_{i, \eps (j)}.
$$

Moreover, an $\Inc$-invariant filtration $\Ic = (I_n)_{n \in \N}$ of ideals
$I _n\subset \Pb_n$ (see, e.g., \cite{HS} or \cite[Definition 5.1]{NR}) corresponds to an $\FIO$-algebra $\Ab$ over $K$, where $\Ab_n = \Pb_n/I_n$.

\item Using the maps $\iota_{m, n}$ with $m \le n$ (see \eqref{eq:iota}), we can form a
 direct system $(\Pb_n, \iota^*_{m, n})$. Then
 one sees that, as $K$-algebras,
$$
K[X] \cong {\displaystyle \lim_{\longrightarrow}}\, \Pb_n.
$$
\end{enumerate}
\end{ex}

The last construction can be considerably generalized.

\begin{defn}
\label{def:limit}
Given any $\FIO$-algebra $\Ab$ over $K$, define its \emph{colimit}
$$
K[\Ab] = {\displaystyle \lim_{\longrightarrow}}\, \Ab_n \in K\text{-}\Alg
$$
using the direct system $(\Ab_n, \iota^*_{m, n})$ with maps $\iota_{m,n}$ as introduced in \eqref{eq:iota}.

Similarly, one defines the colimit $K[\Ab]$ for an $\FI$-algebra $\Ab$.
\end{defn}

\begin{rem}
\
\begin{enumerate}
\item
In the case $\Ab = \Pb$, the above colimit has been studied, for example, in \cite{Draisma, HS, NR} by using the monoid of increasing maps
\[
\Inc = \{\pi\colon \N \to \N \; | \; \pi (i) < \pi (i+1) \text{ for all } i \ge 1\}.
\]
The colimit $K[\Pb] \cong K[X]$ naturally admits an $\Inc$-action induced by $\pi \cdot x_{i, j} = x_{i, \pi (j)}$ for any $\pi \in \Inc$. We will see below that every colimit $K[\Ab]$ admits an $\Inc$-action that is compatible with the $\FIO$-algebra structure of $\Ab$. This close relation is one of the main motivations of our approach.

\item
Similarly, considering $\Pb$ as an $\FI$-algebra, there is compatible $\Sym(\infty)$-action on its colimit $K[X]$. It is induced by $\pi \cdot x_{i, j} = x_{i, \pi (j)}$ with $\pi \in \Sym(\infty)$. Here $\Sym(\infty)$ denotes the group $\bigcup_{n\in \N} \Sym (n)$, where the symmetric group $\Sym (n)$ on $n$ letters is naturally embedded into $\Sym (n+1)$ as the stabilizer of $\{n+1\}$.

\end{enumerate}
\end{rem}

For introducing an $\Inc$-action or an $\Sym (\infty)$-action on arbitrary colimits, we need some further notation.

\begin{defn}
\label{def:induced-maps-by-pi}
\
\begin{enumerate}
\item
Given a map $\pi \in \Inc$, denote by $\eps_{\pi, m}$ the map
\[
\eps_{\pi, m}\colon [m] \to [\pi (m)], \ j \mapsto \pi (j).
\]
It is order-preserving and injective.
\item
For $\pi \in \Sym (\infty)$ and $m \in \N$, let $l_m$ be the least integer such that $\pi ([m]) \subseteq [l_m]$. Define a map
\[
\eps_{\pi, m}\colon [m] \to [l_m], \ j \mapsto \pi (j).
\]
It is injective.
\end{enumerate}
\end{defn}

One easily checks the following identities.

\begin{lem}
\label{lem:map-indentities}
\
\begin{enumerate}
\item For every $\pi \in \Inc$ and any positive integers $m \le k$, one has
\[
\iota_{\pi (m), \pi (k)} \circ \eps_{\pi, m} = \eps_{\pi, k} \circ \iota_{m, k}.
\]

\item
For every $\pi \in \Sym (\infty)$ and any positive integers $m \le k$, one has
\[
\iota_{l_m, l_k} \circ \eps_{\pi, m} = \eps_{\pi, k} \circ \iota_{m, k}.
\]
\end{enumerate}
\end{lem}

For any $a \in \Ab_m$, denote by $[a]$ its class in the colimit $K[\Ab] ={\displaystyle \lim_{\longrightarrow}}\, \Ab_n$.

\begin{prop}
\label{prop:action-on-limits}
\
\begin{enumerate}
\item Let $\Ab$ be an $\FIO$-algebra. Then there is an $\Inc$-action on $K[\Ab]$ defined by
 \[
\pi \cdot [a] = [\eps_{\pi, m}^* (a)],
\]
where $a \in \Ab_m$ and $\pi \in \Inc$.
\item
 Let $\Ab$ be an $\FI$-algebra. Then there is a $\Sym (\infty)$-action on $K[\Ab]$ defined by
 \[
\pi \cdot [a] = [\eps_{\pi, m}^* (a)],
\]
where $a \in \Ab_m$ and $\pi \in \Sym (\infty)$.
 \end{enumerate}
\end{prop}

\begin{proof}
(i) First we show that the stated assignment gives a well-defined map.
Consider any $a \in \Ab_m$ and $b \in \Ab_n$ with $[a] = [b]$, i.e., there is an integer $k \ge m, n$ such that $\iota_{m, k}^* (a) = \iota_{n, k}^* (b)$. Using Lemma~\ref{lem:map-indentities}(i), we get
\[
\iota_{\pi (m), \pi (k)} ^* (\eps_{\pi, m}^* (a)) = \eps_{\pi, k}^* ( \iota_{m, k}^* (a)) = \eps_{\pi, k}^* (\iota_{n, k}^* (b)) = \iota_{\pi (n), \pi (k)} ^* (\eps_{\pi, n}^* (b)),
\]
which shows $\pi \cdot [a] = \pi \cdot [b]$. The map yields an action on $K[\Ab]$ because $\eps_{\tilde{\pi} \circ \pi, m} = \eps_{\tilde{\pi}, \pi (m)} \circ \eps_{\pi, m}$.

(ii) is shown similarly using Lemma~\ref{lem:map-indentities}(ii).
\end{proof}

For comparing structures, we need the following technical result.

\begin{lem}
\label{lem:compatible-structures}
\
\begin{enumerate}
\item
Let $\Ab$ be an $\FIO$-algebra. Given any $\eps \in \Hom_{\FIO} ([m], [n])$,
there is some $\pi_{\eps} \in \Inc$ such that
\[
[\eps^* (a)] = \pi_{\eps} \cdot [a] \quad \text{ for every } a \in \Ab_m.
\]
\item
Let $\Ab$ be an $\FI$-algebra. Given any $\eps \in \Hom_{\FI} ([m], [n])$, there is some $\pi_{\eps} \in \Sym (n)$ such that
\[
[\eps^* (a)] = \pi_{\eps} \cdot [a] \quad \text{ for every } a \in \Ab_m.
\]

\end{enumerate}
\end{lem}

\begin{proof}
(i) Define a map $\pi_\eps \in \Inc$ by
\begin{align}
\label{eq:induced-by-eps}
\pi_\eps (j) =
\begin{cases}
\eps(j)&\text{if } j\in [m],\\
\eps(m)+ j-m & \text{if } j>m.
\end{cases}
\end{align}
Using Definition~\ref{def:induced-maps-by-pi}(i) and $n \ge \eps (m)$, one checks that $\eps = \iota_{\eps (m), n} \circ \eps_{\pi_{\eps}, m}$. Thus, we get $\pi_{\eps} \cdot [a] = [ \eps_{\pi_{\eps}, m}^* (a)] = [\eps^* (a)]$, as desired.

(ii) Choose $\pi_{\eps} \in \Sym (n)$ with $\pi_{\eps} (j) = \eps (j)$ for every $j \in [m]$. Using Definition~\ref{def:induced-maps-by-pi}(ii) and $n \ge l_m$, one checks that $\eps = \iota_{l_m, n} \circ \eps_{\pi_{\eps}, m}$. Now the claim follows as in (i).
\end{proof}

It follows immediately from Lemma~\ref{lem:compatible-structures} that the structure of $\Ab$ and the above action on its colimit $K[\Ab]$ are compatible in the following sense.

\begin{cor}
Let $\Ab$ be an $\FIO$-algebra (or an $\FI$-algebra, respectively).
\begin{enumerate}
\item Given any $\pi \in \Inc$ (or $\pi \in \Sym (\infty)$, respectively) and any $m \in \N$, one has
\[
\pi \cdot [a] = \pi_{\eps_{\pi}} \cdot [a] \quad \text{ for every } \pi_{\eps} \text{ as in Lemma~\ref{lem:compatible-structures} and every } a \in \Ab_m.
\]

\item Given any $\eps \in \Hom_{\FIO} ([m], [n])$ (or $\eps \in \Hom_{\FI} ([m], [n])$), one has
\[
[\eps^* (a)] = [\eps_{\pi_{\eps}}^* (a)] \quad \text{ for every } \pi_{\eps} \text{ as in Lemma~\ref{lem:compatible-structures} and every } a \in \Ab_m.
\]
\end{enumerate}

\end{cor}

We now discuss finite generation of algebras.
If $\Ab$ is an $\FIO$- or an $\FI$-algebra over $K$, then a \emph{subalgebra} $\Bb$ of $\Ab$ is naturally defined by identifying subalgebras $\Bb_S$ of $\Ab_S$ for each (totally ordered) finite set $S$ such that
for $\eps^*\colon \Ab_S \to \Ab_T$ we have $\eps^*(\Bb_S)\subseteq \Bb_T$
for any morphism $\eps\colon S \to T$.

\begin{defn}
Let $\Ab$ be an $\FIO$-algebra (or $\FI$-algebra over $K$, respectively).
\begin{enumerate}
\item
$\Ab$ is called \emph{finitely generated}, if there exists
a finite subset $G \subset \coprod_{n\geq 0} \Ab_n$
which is not contained in any proper subalgebra of $\Ab$.

\item
$\Ab$ is called \emph{generated in degrees $\leq d$}
(resp.\ \emph{generated in degree $d$})
if there exists a set $G \subseteq \coprod_{0\leq n\leq d} \Ab_n$
(resp.\ $G \subseteq \Ab_d$)
which is not contained in any proper subalgebra of $\Ab$.
\end{enumerate}

In both cases, $G$ is called a \emph{generating set} of $\Ab$.
\end{defn}

\begin{rem}
Let $\Ab$ be an $\FIO$-algebra and consider a finite subset $G=\{a_1,\dots,a_k\}$, where $a_i\in \Ab_{n_i}$
for $i=1,\dots,k$. Then $\Ab$ is finitely generated by $G$ if and only if one has, for every totally ordered finite set $T$,
\[
\Ab_T=K\bigl[\eps^*(a_i): \eps\in \Hom_{\FIO} ([n_i], T)\bigr].
\]
The analogous statement is true for an $\FI$-algebra.
\end{rem}
For the following observation we use the notation
$$
\Ab_{n} \cap H:= \phi_n^{-1}(H)
\text{ for a subset } H \subseteq K[\Ab],
$$
where $\phi_n$ is the natural homomorphism
from $\Ab_n$ to the colimit $K[\Ab]$ as introduced in Definition~\ref{def:limit}.

\begin{prop}
\label{prop:compare-gen}
Let $\Ab$ be an $\FIO$-algebra (or $\FI$-algebra over $K$, respectively).
Let $\Pi=\Inc$ (or $\Pi=\Sym(\infty)$, respectively).

If $\Ab$ is finitely generated (in degrees $\leq d$), then
$K[\Ab]$ is a finitely generated $K$-algebra up to $\Pi$-action, i.e.,
it is generated as a $K$-algebra by the $\Pi$-orbits of finitely many elements
which have representatives in $\coprod_{n \le d} \Ab_n \cap K[\Ab]$.
\end{prop}

\begin{proof}
We prove this for an $\FIO$-algebra $\Ab$.
For an $\FI$-algebra the argument is similar.
Assume $\Ab$ is generated by $G=\{a_1,\dots,a_k\}$ where $a_i\in \Ab_{n_i}$
for $i=1,\dots,k$. Thus, for every totally ordered finite set $T$, we have
$$
\Ab_T=
K\bigl[\eps^*(a_i): \eps\in \Hom_{\FIO} ([n_i], T)\bigr].
$$
Consider any $[a]\in K[\Ab]$, where $a\in \Ab_n$.
Then $a$ can be written as a polynomial
in $\eps^*(a_i)$ with coefficients in $K$ where $\eps\in \Hom_{\FIO} ([n_i], [n])$.
For any fixed map $\eps\in \Hom_{\FIO} ([n_i], [n])$, there is a map $\pi_\eps \in \Inc$ such that $\pi_\eps \cdot [a_i] = [\eps^*(a_i)]$ (see Lemma~\ref{lem:compatible-structures}).
It follows that $[a]$ is a polynomial
in certain elements of the $\Inc$-orbits of the $[a_i]$ with coefficients in $K$.
Thus, $K[\Ab]$ is a finitely generated $K$-algebra up to $\Inc$-action, as desired.

The arguments also imply the additional statement about degrees.
\end{proof}

\begin{rem}
The converse fails in general. Consider, for example, the $\FIO$-algebra $\Ab$, where $\Ab_S = K[X]$ (see Example~\ref{exa:OI-alg}) if $S$ is a totally ordered finite set with $|S| = 1$, $\Ab_S = K$ if $|S| \neq 1$, and with obvious maps. Then $\Ab$ is not finitely generated, but $K[\Ab] \cong K$ is.

It would be interesting to find additional conditions which imply that the converse of Proposition~\ref{prop:compare-gen} is true.
\end{rem}

We now introduce an important class of algebras.

\begin{defn}
Let $d \ge 0$ be an integer.
\begin{enumerate}
\item
Define a functor $\XO{d}\colon \FIO \to K$-$\Alg$ by letting
\[
\XO{d}_S = K\bigl[x_{\pi} \; : \; \pi \in \Hom_{\FIO} ([d], S)\bigr]
\]
be the polynomial ring over $K$ with variables $x_{\pi}$,
and, for $\eps \in \Hom_{\FIO} (S, T)$, by defining
 \[
\XO{d} (\eps)\colon \XO{d}_S \to \XO{d}_T
 \]
as the $K$-algebra homomorphism given by mapping $x_{\pi}$ onto $x_{\eps \circ \pi}$.

A \emph{polynomial $\FIO$-algebra over $K$} is an $\FIO$-algebra that is isomorphic to a tensor product
 $\Xb = \bigotimes_{\lambda \in \Lambda} \XO{d_{\lambda}}$, where each $\Xb_S$ is a tensor product of rings
 $\XO{d_{\lambda}}_S$ over $K$.
\item
Ignoring orders, we similarly define an $\FI$-algebra $\XI{d}$ over $K$ and a \emph{polynomial $\FI$-algebra over $K$}.
\end{enumerate}
\end{defn}

\begin{rem}
\label{rem:poly-alg}
\
\begin{enumerate}
\item
Note that, for each $d \ge 1$, the algebra $\XO{d}$ as well as $\XI{d}$ is generated in degree $d$ by the variable $x_{\id_{[d]}}$ (see Lemma~\ref{lem:decompose}). For $d = 0$, these algebras are generated in degree zero by the identity of $K$. In particular, $\XO{0}_n = \XI{0}_n = K$ for every $n$.
\item
Identifying $\eps \in \Hom_{\FIO} ([d], S)$ with its image $s_1 < s_2 < \cdots < s_d$ in $S$, we get, for example, $\XO{1}_S = K[x_i \; : \; i \in S]$ and $\XO{2}_S = K[x_{i, j} \; : \; i, j \in S, \ i < j]$.
Thus we obtain for the colimits, as $K$-algebras,
\[
K[\XO{1}] \cong K[x_i \; : \; i \in \N] \cong K[\XI{1}]
\]
and
\[
K[\XO{2}] \cong K[x_{i, j}\; : \; i, j \in \N, \ i < j].
\]
In contrast, $\XI{2}_S = K[x_{i, j} \; : \; i, j \in S, \ i \neq j]$ and
$
K[\XI{2}] \cong K[x_{i, j}\; : \; i, j \in \N, \ i \neq j].
$
\item
Observe that $(\XO{1})^{\otimes c}$ is isomorphic to the $\FIO$-algebra $\Pb$ considered in Example~\ref{exa:OI-alg}.
\item
Notice that the $\FIO$-algebra induced by $\XI{d}$ is isomorphic to $\XO{d}$ if and only if $d \in \{0, 1\}$. Furthermore, the algebras $\XI{d}_n$ and $\XO{d}_n$ are isomorphic for every integer $n \ge 0$ if and only if $d \in \{0, 1\}$.
\end{enumerate}
\end{rem}

\begin{prop}
\label{prop:fg-FOI-algebra}
Let $\Ab$ be an $\FIO$-algebra over $K$. The following statements are equivalent:
\begin{enumerate}
\item $\Ab$ is finitely generated (in degrees $\le d$);

\item there is a surjective natural transformation $\XO{d_1} \otimes_K \cdots \otimes_K \XO{d_k} \to \Ab$ for some integers $d_1,\ldots,d_k \ge 0$ (with all $d_i \le d$).
\end{enumerate}

The analogous equivalence is true for every $\FI$-algebra over $K$.
\end{prop}

\begin{proof}
We prove this for $\FIO$-algebras, leaving the $\FI$ case to the interested reader.

By Remark \ref{rem:poly-alg}(i), the algebra $\XO{d_1} \otimes_K \cdots \otimes_K \XO{d_k}$ has $k$ generators of degrees $d_1,\ldots,d_k$. Thus, (ii) implies (i).

Conversely, a set $G = \{a_1,\ldots,a_k\}$ with $a_i \in \Ab_{d_i}$ determines
canonically a natural transformation $\XO{d_1} \otimes_K \cdots \otimes_K \XO{d_k} \to \Ab$. Its image is the
$\FIO$-algebra generated by $G$.
\end{proof}

\begin{rem}
\
\begin{enumerate}
\item
A \emph{$\Z$-graded $\FIO$-algebra} is an $\FIO$-algebra $\Ab$ over $K$ such that every $\Ab_S$ is a $\Z$-graded $K$-algebra and every map $\Ab (\eps)\colon \Ab_S \to \Ab_T$ is a graded homomorphism of degree zero.
We will refer to it simply as a graded $\FIO$-algebra. Similarly, we define a graded $\FI$-algebra.
\item
Note that the polynomial algebras $\XO{d}$ and $\XI{d}$ are naturally graded. Furthermore, there are analogous results for graded algebras for each of the above results.
\item
More generally, one can consider $\FIO$- and $\FI$-algebras with a more general grading semigroup. We leave this to the interested reader.
\end{enumerate}
\end{rem}

Consider any $\Z$-graded ring $R = \oplus_{j \in \Z} [R]_j$ and fix an integer $e \ge 1$. The subring $\oplus_{j \in \Z} [R]_{j e}$ is called the \emph{$e$-th Veronese subring} of $R$ and denoted by $R^{(e)}$. There is an analogous construction for $\FIO$- and $\FI$-algebras.

\begin{ex}
\label{exa:Veronese}
Let $\Ab$ be a $\Z$-graded $\FIO$-algebra (or $\FI$-algebra, respectively). For any $e \in \N$, let $\Ab^{(e)}$ be the subalgebra of $\Ab$ induced by setting
$$\big ( \Ab^{(e)} \big )_n = (\Ab_n)^{(e)} \text{ for every } n \in \N_0.
$$
It is called the \emph{$e$-th Veronese subalgebra} of $\Ab$. It also is a $\Z$-graded algebra. Note that $\Ab^{(e)}$ is generated as $\FIO$-algebra by the set
$ \coprod_{n \ge 0} \bigcup_{j \ge 0} [\Ab_n]_{j e}$.
\end{ex}


\section{$\FI$- and $\FIO$-modules}
\label{sec:fio-modules}

For introducing our main objects of interest, we denote by $K$-$\Mod$ the category of $K$-modules.

\begin{defn}
\label{def:OI-module}
\
\begin{enumerate}
\item
Let $\Ab$ be an $\FIO$-algebra over $K$. The category of \emph{$\FIO$-modules over $\Ab$} is denoted by $\FIO$-$\Mod (\Ab)$. Its objects are covariant functors $\M \colon \FIO \to K$-$\Mod$ such that,
\begin{itemize}
\item[(1)] for any finite totally ordered set $S$, the $K$-module $\M_S = \M(S)$ is also an $\Ab_S$-module, \ and

\item[(2)]
for any morphisms $\widetilde{\eps}\colon \widetilde{S} \to S$, $\eps\colon S \to T$ and any $a \in \Ab_{\widetilde{S}}$, the following diagram is commutative
\begin{equation*}
\xy\xymatrixrowsep{1.8pc}\xymatrixcolsep{1.8pc}
\xymatrix{
\M_S \ar @{->}[r]^-{\M (\eps)} \ar @{->}[d]_-{\cdot \Ab(\widetilde{\eps}) (a)} & \M_T \ar @{->}[d]^-{\cdot \Ab(\eps \circ \widetilde{\eps})(a)} \\
\M_S \ar @{->}[r]^-{\M (\eps)} & \M_T.
}
\endxy
\end{equation*}
Here the vertical maps are given by multiplication by the indicated elements.
\end{itemize}
The morphisms of $\FIO$-$\Mod (\Ab)$ are natural transformations $F \colon \M \to \Nb$ such that, for each totally ordered finite set $S$, the map $\M_S \stackrel{F(S)}{\longrightarrow} \Nb_S$ is an $\Ab_S$-module homomorphism.
\item
Ignoring orders, we define similarly the category $\FI$-$\Mod (\Ab)$ of $\FI$-modules over an $\FI$-algebra $\Ab$.

Its objects are functors from $\FI$ to $K$-$\Mod$ and its morphisms are natural transformations satisfying conditions analogous to those above.
\end{enumerate}
\end{defn}

\begin{rem}
\label{rem:abelian-cat-submodules}
\
\begin{enumerate}
\item
Given any morphisms $\eps_1\colon S_1 \to S$, $\eps_2\colon S_2 \to S$, $ \eps\colon S \to T$ and any elements $a \in \Ab_{S_1}$, $q \in \M_{S_2}$, the assumptions imply, for example,
\[
\M (\eps) \bigg ( \Ab (\eps_1) (a) \cdot \M(\eps_2) (q) \bigg ) = \Ab (\eps \circ \eps_1) (a) \cdot \M (\eps \circ \eps_2) (q).
\]
\item
In the case of \emph{constant coefficients}, where $\Ab = \XI{0}$ is the ``constant'' $\FI$-algebra over $K$, that is, $\Ab_n = K$ for each $n$, the category $\FI$-$\Mod (\Ab)$ is exactly the category of $\FI$-modules over $K$ as, for example, studied in \cite{CE, CEF, CEFN}. We are mainly interested in the case where the coefficients, that is, the $K$-algebras $\Ab_n$ vary with $n$. This additional structure is crucial for our results.
\item
The categories $\FIO$-$\Mod (\Ab)$ and $\FI$-$\Mod (\Ab)$ inherit the structure of an abelian category from $K$-$\Mod$, with all concepts such as subobject, quotient object, kernel, cokernel, injection, and surjection being defined ``pointwise'' from the corresponding concepts in $K$-$\Mod$ (see \cite[A.3.3]{W}).
\item
Notice that every $\FI$-module over an $\FI$-algebra may also be considered as an $\FIO$-module over the induced $\FIO$-algebra. Again, we typically will use the same notation for these modules.
\end{enumerate}
\end{rem}

\begin{ex}
Consider the polynomial ring $K[X]$ and the $\FIO$-algebra $\Pb$ as defined in Example \ref{exa:OI-alg}.
\begin{enumerate}
\item
An $\Inc$-invariant filtration $\Ic = (I_n)_{n \in \N}$ of ideals $I_n \subset \Pb_n$ is the same as an ideal (submodule) of the $\FIO$-algebra $\Pb$.
\item
Similarly, an $\Sym (\infty)$-invariant filtration $\Ic = (I_n)_{n \in \N}$
(see, e.g., \cite{HS} or \cite[Theorem 7.8]{NR}) is the same as an ideal of $\Pb$, considered
as an $\FI$-algebra. Such a filtration is also an $\Inc$-invariant filtration
(see Remark \ref{rem:abelian-cat-submodules}).
\end{enumerate}
\end{ex}

\begin{defn}
\label{def:limit-module}
Let $\M$ be an $\FIO$-module over an $\FIO$-algebra $\Ab$. Define its \emph{colimit}
$$
\lim \M = {\displaystyle \lim_{\longrightarrow}}\, \M_n \in K\text{-}\Mod
$$
using the direct system $(\M_n, \M (\iota_{m, n}))$ with maps $\iota_{m,n}$ given in \eqref{eq:iota}.

Similarly, one defines the colimit $\lim \M$ for an $\FI$-algebra $\M$.
\end{defn}

For any $q \in \M_m$, denote by $[q]$ its class in the colimit $\lim \M$.

\begin{lem}
If $\M$ is an $\FIO$-module over an $\FIO$-algebra $\Ab$, then $\lim \M$ is a module over $K[\Ab]$ with scalar multiplication defined by
\[
[a] \cdot [q] = [\Ab (\iota_{m, k}) (a) \cdot \M (\iota_{n, k}) (q)],
\]
where $a \in \Ab_m$, $q \in \M_n$, and $k = \max\{m, n\}$.

Similarly, we have that $\lim \M$ is a module over $K[\Ab]$ if $\M$ is an $\FI$-module over an $\FI$-algebra $\Ab$.
\end{lem}

\begin{proof}
This is a routine argument. We leave the details to the interested reader.
\end{proof}

The above colimits also admit certain actions. Recall that the maps $\eps_{\pi, m}$ are introduced in Definition~\ref{def:induced-maps-by-pi}.

\begin{prop}
\
\begin{enumerate}
\item Let $\M$ be an $\FIO$-module. Then there is an $\Inc$-action on $\lim \M$ defined by
 \[
\pi \cdot [q] = [\M (\eps_{\pi, m}) (q)],
\]
where $q \in \M_m$ and $\pi \in \Inc$.
\item
 Let $\M$ be an $\FI$-module. Then there is a $\Sym (\infty)$-action on $\lim \M$ defined by
 \[
\pi \cdot [q] = [\M (\eps_{\pi, m}) (q)],
\]
where $q \in \M_m$ and $\pi \in \Sym (\infty)$.
 \end{enumerate}
\end{prop}

\begin{proof}
The argument is completely analogous to the one for Proposition \ref{prop:action-on-limits}.
\end{proof}

The action on colimits and their scalar multiplication are compatible in the following sense.

\begin{lem}
Let $\M$ be an $\FIO$-module over an $\FIO$-algebra $\Ab$ (or an $\FI$-module over an $\FI$-algebra $\Ab$,
respectively). For any $a \in \Ab_m$, $q \in \M_s$ and any $\pi \in \Inc$ (or $\pi \in \Sym (\infty)$, respectively), one has
\[
\pi \cdot \big ( [a] \cdot [q] \big ) = \big (\pi \cdot [a] \big) \cdot \big ( \pi \cdot [q] \big ).
\]
\end{lem}

\begin{proof}
Set $t = \max \{m, s\}$. The definitions imply
\begin{eqnarray*}
\pi \cdot \big ( [a] \cdot [q] \big ) & = & \big [\M (\eps_{\pi, t}) \big (\Ab (\iota_{m, t} (a) \cdot \M (\iota_{s, t}) (q) \big ) \big] \\
& = & \big [ \Ab (\eps_{\pi, t} \circ \iota_{m, t}) (a) \cdot \M(\eps_{\pi, t} \circ \iota_{s, t}) (q) \big ] \\
& = & \big [ \Ab (\iota_{\pi (m), \pi(t)} \circ \eps_{\pi, m} ) (a) \cdot \M(\iota_{\pi (s), \pi(t)} \circ \eps_{\pi, s}) (q) \big ] \\
& = & \big [ \Ab (\eps_{\pi, m} ) (a) \cdot \M(\eps_{\pi, s}) (q) \big ]
 = \big (\pi \cdot [a] \big) \cdot \big ( \pi \cdot [q] \big ),
\end{eqnarray*}
where we used Remark~\ref{rem:abelian-cat-submodules}(i) for the second equality. The third equality is a consequence of  $\eps_{\pi, t} \circ \iota_{m, t} = \iota_{\pi (m), \pi(t)} \circ \eps_{\pi, m}$.
\end{proof}

As for algebras (see Lemma~\ref{lem:compatible-structures}), one obtains the following observation.

\begin{lem}
\label{lem:compatible-structures-mod}
\

\begin{enumerate}
\item Let $\M$ be an $\FIO$-module over an $\FIO$-algebra $\Ab$. Given any $\eps \in \Hom_{\FIO} ([m], [n])$, there is some $\pi_{\eps} \in \Inc$ such that
\[
[\M (\eps) (q)] = \pi_{\eps} \cdot [q] \quad \text{ for every } q \in \M_m.
\]

\item Let $\M$ be an $\FI$-module over an $\FI$-algebra $\Ab$. Given any $\eps \in \Hom_{\FI} ([m], [n])$, there is some $\pi_{\eps} \in \Sym (n)$ such that
\[
[\M (\eps) (q)] = \pi_{\eps} \cdot [q] \quad \text{ for every } q \in \M_m.
\]

\end{enumerate}
\end{lem}

It follows that the structure of $\M$ and the above action on its colimit $\lim \M$ are compatible in the following sense.

\begin{cor}
Let $\M$ be an $\FIO$-module over an $\FIO$-algebra $\Ab$ (or an $\FI$-module over an $\FI$-algebra $\Ab$, respectively).
\begin{enumerate}
\item Given any $\pi \in \Inc$ (or $\pi \in \Sym (\infty)$, respectively) and any $m \in \N$, one has
\[
\pi \cdot [q] = \pi_{\eps_{\pi}} \cdot [q] \quad \text{ for every } \pi_{\eps} \text{ as in Lemma~\ref{lem:compatible-structures-mod} and every } q \in \M_m.
\]

\item Given any $\eps \in \Hom_{\FIO} ([m], [n])$ (or $\eps \in \Hom_{\FI} ([m], [n])$), one has
\[
[\M (\eps) (q)] = [\M (\eps_{\pi_{\eps}}) (q)] \quad \text{ for every } \pi_{\eps} \text{ as in Lemma~\ref{lem:compatible-structures-mod} and every } q \in \M_m.
\]
\end{enumerate}

\end{cor}

We now discuss finite generation of modules.

\begin{defn}
Let $\M$ be an $\FIO$-module over an $\FIO$-algebra $\Ab$ (or an $\FI$-module over an $\FI$-algebra $\Ab$, respectively).
\begin{enumerate}
\item
$\M$ is called \emph{finitely generated}, if there exists
a finite subset $G \subset \coprod_{n\geq 0} \M_n$
which is not contained in any proper submodule of $\M$.

\item
$\M$ is called \emph{generated in degrees $\leq d$}
(resp.\ \emph{generated in degree $d$})
if there exists a set $G \subseteq \coprod_{0\leq n\leq d} \M_n$
(resp.\ $G \subseteq \M_d$)
which is not contained in any proper submodule of $\M$.
\end{enumerate}
In both cases, $G$ is called a \emph{generating set} of $\M$.
\end{defn}

\begin{rem}
\label{rem:finite-gen-module}
Let $\M$ be an $\FIO$-module over an $\FIO$-algebra $\Ab$.
Then $\M$ is generated by $G \subseteq \coprod_{0\leq n\leq d} \M_n$ if and only if one has, for every integer $k \ge 0$,
\[
\M_k= \big \la \M (\eps)(q) \; : \; q \in G \cap \M_n, \ \eps\in \Hom_{\FIO} ([n], [k]) \big \ra.
\]

The analogous statement is true for an $\FI$-module.
\end{rem}

There is an alternate  description of generation up to degree $d$.

\begin{defn}
\

\begin{enumerate}
\item Let $\M$ be an $\FIO$-module over an $\FIO$-algebra $\Ab$. It is said that $\M$ \emph{stabilizes} if there is an integer $r$ such that, as $\Ab_n$-modules, one has
\[
\big \la\M (\eps) (\M_r) \; : \;  \eps \in  \Hom_{\FIO} ([r], [n]) \big \ra = \M_n \text{ whenever } r \le n.
\]

\item An $\FI$-module over an $\FI$-algebra $\Ab$ \emph{stabilizes} if there is an integer $r$ such that, as $\Ab_n$-modules, one has
\[
\big \la\M (\eps) (\M_r) \; : \;  \eps \in  \Hom_{\FI} ([r], [n]) \big \ra = \M_n \text{ whenever } r \le n.
\]

\end{enumerate}

In both cases, the least integer $r \ge 1$ with this property is said to be the \emph{stability index} $\ind (\M)$ of $\M$.
\end{defn}

\begin{lem}
\label{lem:gen-vs-stability}
Let $\M$ be an $\FIO$-module over an $\FIO$-algebra $\Ab$ (or an $\FI$-module over an $\FI$-algebra $\Ab$, respectively). The following statements are equivalent:
\begin{enumerate}
\item $M$ is generated in degree $\le d$;

\item $M$ stabilizes and $\ind (\M) \le d$.
\end{enumerate}
\end{lem}

\begin{proof}
Taking $G = \coprod_{0\leq n\leq d} \M_n$, (ii) implies (i) by Remark~\ref{rem:finite-gen-module}.

Conversely, possibly replacing the given generating set by $\coprod_{0\leq n\leq d} \M_n$, we conclude by Lemma~\ref{lem:decompose} and Remark~\ref{rem:finite-gen-module}.
\end{proof}

Now we compare generators of a module and its colimit.
For this we need the notation
$$
\M_{n} \cap H:= \phi_n^{-1}(H)
\text{ for a subset } H \subseteq \lim \M,
$$
where $\phi_n$ is the natural homomorphism
from $\M_n$ to $\lim \M$ as introduced in Definition~\ref{def:limit-module}.

\begin{prop}
\label{prop:fin-gen-to-limit}
Let $\M$ be an $\FIO$-module over an $\FIO$-algebra $\Ab$ (or $\FI$-module over an $\FI$-algebra $\Ab$, respectively).
Let $\Pi=\Inc$ (or $\Pi=\Sym(\infty)$, respectively).

If $\M$ is finitely generated (in degrees $\leq d$), then $\lim \M$ is a finitely generated $K[\Ab]$-module up to $\Pi$-action, i.e.,
it is generated as a $K[\Ab]$-module by the $\Pi$-orbits of finitely many elements
(which have representatives in $\coprod_{n \le d} \M_n \cap \lim \M$).
\end{prop}

\begin{proof}
We prove this for an $\FIO$-module $\M$.
For an $\FI$-module the argument is similar.

Assume $\M$ is generated by $G=\{q_1,\dots,q_k\}$ where $a_i\in \Ab_{n_i}$
for $i=1,\dots,k$. Thus, for every integer $k \ge 0$,
\[
\M_k= \big \la \M (\eps)(q_i) \; : \; \eps\in \Hom_{\FIO} ([n_i], [k]), \ i = 1,\ldots,k \big \ra.
\]
Consider any $[q]\in \lim \M$, where $q \in \M_k$.
Then $q$ can be written as a linear combination of
$\M(\eps) (q_i)$ with coefficients in $\Ab_k$ where $\eps\in \Hom_{\FIO} ([n_i], [k])$. For any fixed map
$\eps\in \Hom_{\FIO} ([n_i], [k])$, there is a map $\pi_\eps \in \Inc$ such that $\pi_\eps \cdot [q_i] = [\M (\eps)(q_i)]$ (see
Lemma~\ref{lem:compatible-structures-mod}).

It follows that $[q]$ is a linear combination of certain elements of the $\Inc$-orbits of the $[q_i]$ with coefficients in $\Ab_k$.
Thus, $\lim \M$ is a finitely generated up to $\Inc$-action, as desired.
The arguments also imply the additional statements about degrees.
\end{proof}

\begin{rem}
\
\begin{enumerate}
\item
Again, the converse fails in general.
\item
It is typically difficult to show directly that $\lim \M$ is finitely generated. Indeed, \cite[Theorem 1.1]{HS} states that every $\FI$-ideal $\Ib$ of $\Pb$ (see Example~\ref{exa:OI-alg}) has a finitely generated colimit.
\end{enumerate}
\end{rem}

We now define a class of $\FIO$-modules in order to discuss freeness.

\begin{defn}
\label{def:free}
\
\begin{enumerate}
\item
For an $\FIO$-algebra $\Ab$ over $K$ and an integer $d \ge 0$, let $\Fo{d}$ be the $\FIO$-module over $\Ab$ defined by
\[
\Fo{d}_S = \oplus_{\pi} \Ab_S e_{\pi} \cong (\Ab_S )^{\binom{|S|}{d}},
\]
where $S$ is a totally ordered finite subset and the sum is taken over all $\pi \in \Hom_{\FIO} ([d], S)$, and
\[
\Fo{d}(\eps)\colon \Fo{d}_S \to \Fo{d}_T, \ a e_{\pi} \mapsto \eps^*(a) e_{\eps \circ \pi},
\]
where $a \in \Ab_S$ and $\eps\colon S \to T$ is an $\FIO$ morphism.

A \emph{free $\FIO$-module} over $\Ab$ is an $\FIO$-module that is isomorphic to
$\bigoplus_{\lambda \in \Lambda} \Fo{d_{\lambda}}$.
\item
Ignoring orders, we similarly define an $\FI$-module $\Fi{d}$ over an $\FI$-algebra $\Ab$ and a \emph{free $\FI$-module} over $\Ab$.
\end{enumerate}
\end{defn}

\begin{rem}
\label{rem:free}
\
\begin{enumerate}
\item
If $\Ab$ is an $\FIO$-algebra, then the module $\Fo{0}$ is isomorphic to $\Ab$ considered as an $\FIO$-module over itself. Two modules $\Fo{d_1}$ and $\Fo{d_2}$ are isomorphic if and only if $d_1 = d_2$.
\item
Note that $\Fo{d}$ is not isomorphic to a direct sum of copies of $\Ab$ if $d \ge 1$. Thus, there are free $\FIO$-modules over $\Ab$ other than direct sums of copies of $\Ab$.
\item
For every $\FI$-algebra $\Ab$, there is a free $\FI$-module $\Fi{d}$ and an induced free $\FIO$-module. The latter is obtained by considering $\Ab$ as an $\FIO$-algebra. Notice that the $\Ab_n$-modules $\Fo{d}_n$ and $\Fi{d}_n$ are isomorphic for every integer $n \ge 0$ if $d \in \{0, 1\}$. In general, this is false if $d \ge 2$.
\item
More specifically, every free $\FI$-module $\Fi{d}$ over $(\XI{k})^{\otimes c}$ for any $d, k \in \N_0$ and $c \in \N$ is generated by $e_{\id_{[d]}}$. However, considered as an $\FIO$-module over the $\FIO$-algebra $(\XO{k})^{\otimes c}$, it is generated by $\{e_{\sigma} \; | \; \sigma \in \Sym (d)\} \subset \Fi{d}_d$. Indeed, consider any $\pi \in \Hom_{\FI} ([d], [n])$. There is a unique $\tilde{\pi} \in \Hom_{\FIO} ([d], [n])$ with the same image as $\pi$. Then $\pi = \tilde{\pi} \circ \sigma$ for some $\sigma \in \Sym (d)$.
\end{enumerate}
\end{rem}

\begin{prop}
\label{prop:finite-gen}
Let $\M$ be an $\FIO$-module over an $\FIO$-algebra $\Ab$.

\begin{enumerate}
\item $\M$ is finitely generated if and only if there is a surjection
\[
\bigoplus_{i = 1}^k \Fo{d_{i}} \to \M
\]
for some integers $d_i \ge 0$.
\item
$\M$ is generated in degrees $\le d$ if and only if there is a surjection
\[
\bigoplus_{\lambda \in \Lambda} \Fo{d_{\lambda}} \to \M \quad \text{ with all } d_{\lambda} \le d.
\]
\end{enumerate}

The analogous statements are true for every $\FI$-module $M$ over an $\FI$-algebra $\Ab$.
\end{prop}

\begin{proof}
We prove this for $\FIO$-modules, leaving the $\FI$ case to the interested reader.

Since $\Fo{d}$ is generated by the element $e_{\id_{[d]}} \in \Fo{d}_d$, the module $\bigoplus_{\lambda \in \Lambda} \Fo{d_{\lambda}}$ is finitely generated if $\Lambda$ is finite. It is generated in degree $d$ if $d_{\lambda} \le d$ for all
$\lambda \in \Lambda$. Now the corresponding statements for $\M$ follow which shows necessity.

Conversely, a set $G = \{m_{\lambda} \; | \; \lambda \in \Lambda\}$ with $m_{\lambda} \in \M_{n_{\lambda}}$ determines
canonically a natural transformation $\bigoplus_{\lambda \in \Lambda} \Fo{d_{\lambda}} \to \M$. Its image is the module generated by $G$.
\end{proof}

\begin{rem}
\label{rem:graded-mod}
\
\begin{enumerate}
\item
A \emph{$\Z$-graded $\FIO$-module} is an $\FIO$-module $\M$ over a graded $\FIO$-algebra $\Ab$  such that every $\M_S$ is a graded $\Ab_S$-module and every map $\M (\eps)\colon \M_S \to \M_T$ is a graded homomorphism of degree zero.
We will refer to it simply as a graded $\FIO$-module. Similarly, we define a graded $\FI$-module.
\item
There are analogous results for graded modules for each of the above results.
\item
More generally, one can consider $\FIO$- and $\FI$-modules with a more general grading semigroup. We leave this to the interested reader.
\end{enumerate}
\end{rem}

The category $\FIO$ contains subcategories that provide a framework for studying generalizations of $\Inc^e$-invariant filtrations studied in \cite{NR}.

\begin{rem}
For a non-negative integer $e$, denote by $\FIO^e$ the category whose objects are totally ordered finite sets and whose morphisms are order-preserving, injective maps that map the first $e$ elements of the domain onto the first $e$ elements of the codomain.
Note that $\FIO$ is the category $\FIO^0$. Then $\FIO^e$ algebras and $\FIO^e$-modules are defined analogously to the case $e = 0$. Suitably modified, many results of this paper can be extended to $\FIO^e$-modules for all $e \ge 0$.
\end{rem}


\section{Noetherian Algebras and Modules}
\label{sec:noeth-gen}

We now begin to discuss finiteness results for $\FIO$- and $\FI$-modules.

\begin{defn}
Let $\Ab$ be an $\FIO$-algebra over $K$. An $\FIO$-module $\M$ over $\Ab$ is said to be \emph{noetherian} if every $\FIO$-submodule of $\M$ is finitely generated.
The algebra $\Ab$ is \emph{noetherian} if it is a noetherian $\FIO$-module over itself.

Analogously, we define a \emph{noetherian $\FI$-module} and a \emph{noetherian $\FI$-algebra}.
\end{defn}

The following results are shown as in a module category, using the same arguments (see, e.g., \cite[pages 14--15]{Mat} combined with Remark~\ref{rem:finite-gen-module}).

\begin{prop}
\label{prop:char-noeth-module}
Let $\Ab$ be an $\FIO$-algebra over $K$.
For an $\FIO$-module $\M$ over $\Ab$, the following conditions are equivalent:
\begin{enumerate}
\item
$\M$ is noetherian;
\item
Every ascending chain of $\FIO$-submodules of $\M$
\[
\M_1 \subseteq \M_2 \subseteq \cdots
\]
becomes stationary;
\item
Every non-empty set of $\FIO$-submodules of $\M$ has a maximal element.
\end{enumerate}

The analogous equivalences are also true for $\FI$-modules.
\end{prop}

\begin{cor}
If $\Ab$ is a noetherian $\Z$-graded $\FIO$-algebra (or $\FI$-algebra, respectively), then so is its $e$-th Veronese subalgebra for every $e \in \N$.
\end{cor}

\begin{proof}
Consider any ascending chain of ideals of $\Ab^{(e)}$. Its extension ideals in $\Ab$ also form an ascending chain, which stabilizes by assumption on $\Ab$. Restricting these ideals to $\Ab^{(e)}$ gives the original chain or ideals.
\end{proof}

Using that the categories of $\FIO$-modules and of $\FI$-modules are abelian it also follows (see, e.g., \cite[Theorem 3.1]{Mat}):

\begin{prop}
\label{prop:noeth-along-sequences}
Consider a short exact sequence of $\FIO$-modules (or $\FI$-modules, respectively) over an $\FIO$-algebra $\Ab$ (or an $\FI$-algebra $\Ab$, respectively)
\[
0 \longrightarrow \M' \longrightarrow \M
\longrightarrow \M''\longrightarrow 0.
\]
Then $M$ is a noetherian if and only if $M'$ and $M''$
are noetherian.
\end{prop}

In particular, direct sums of noetherian $\FIO$-modules over $\Ab$ are again noetherian.

\begin{rem}
The above results lead to the question whether every finitely generated $\FIO$-module over a noetherian $\FIO$-algebra $\Ab$ is noetherian and similarly over an $\FI$-algebra Recall that there are free modules over $\Ab$ other than direct sums of copies of $\Ab$ (see  Remark~\ref{rem:free}). By Proposition~\ref{prop:finite-gen}, the question has an affirmative answer if and only if every module $\Fo{d}$ ($d \ge 0$) is noetherian. We will see later that this is indeed the case over $\XI{1}$ and $\XO{1}$.
\end{rem}

\begin{thm}
\label{thm:noeth-and-limits}
Let $\M$ be an $\FIO$-module over an $\FIO$-algebra $\Ab$
(or an $\FI$-module over an $\FI$-algebra $\Ab$, respectively), and let $\Pi = \Inc$ (or $\Pi = \Sym (\infty)$, respectively).

 If $\M$ is noetherian, then $\lim \M$ is \emph{$\Pi$-noetherian}, that is, every $\Pi$-invariant $K[\Ab]$-submodule of $\lim \M$ can be generated by finitely many $\Pi$-orbits.
\end{thm}

\begin{proof}
Let $\Nc \subset \lim \M$ be an $\Inc$-invariant $K[\Ab]$-submodule. We have to show that $\Nc$ can be generated by finitely many $\Inc$-orbits.

To this end, we define a functor $\Nb\colon \FIO \to K$-$\Mod$. For every $k \in \N_0$, set $\Nc_k = \Nc \cap \M_k$. Moreover, we put $\Nb_S = \Nb_{|S|}$ for every totally ordered finite set $S$. Finally, let $\Nb (\eps_{S, T}) = \Nb(\eps_{|S|, |T|})$ for every $\eps_{S, T} \in \Hom_{\FIO} (S, T)$, where $\Nb(\eps_{|S|, |T|})$ is defined by $\Nb(\eps_{|S|, |T|}) (q) = \M(\eps_{|S|, |T|}) (q)$ for $q \in \Nb_{|S|}$. The $\Inc$-invariance of $\Nc$ and Lemma~\ref{lem:compatible-structures-mod} give that $\Nb(\eps_{|S|, |T|}) (q) \in \Nb_{|T|}$. Now it follows that $\Nb$ is an $\FIO$-module.

Note that $\lim \Nb \subseteq \lim \M$. Thus, we get by construction $\lim \Nb = \Nc$. By assumption $\Nb$ is finitely generated. We conclude by Proposition~\ref{prop:fin-gen-to-limit}.

The arguments in the case of an $\FI$-module are analogous.
\end{proof}

\begin{cor}
\label{cor:noeth-limit}
Let $\Ab$ be an $\FIO$-algebra
(or an  $\FI$-algebra, respectively), and let $\Pi = \Inc$ (or $\Pi = \Sym (\infty)$, respectively).

 If $\Ab$ noetherian, then $K[\Ab]$ is $\Pi$-noetherian, that is, every $\Pi$-invariant ideal of $K[\Ab]$ can be generated by finitely many $\Pi$-orbits.
\end{cor}

In both results above, it would be interesting to identify instances in which the converse is true.

We now consider the noetherian property for $\FIO$- and $\FI$-algebras.
As in the classical case, any such noetherian algebra is finitely generated. The converse is not true.

As preparation, we note:

\begin{prop}
Assume $K$ is field. Let $d \ge 0$ be an integer.
Then $\XO{d}$ is noetherian if and only if $d \in \{0, 1\}$.

The analogous statement is true for $\XI{d}$.
\end{prop}

\begin{proof}
The case $d = 0$ is clear.
Let $d = 1$. Note that an ideal $\Ib$ of $\XO{1}$ is essentially the same as an $\Inc$-invariant filtration (see \cite[Definition 5.1]{NR}) by relating any $\FIO$-morphism $\eps$ to $\pi_{\eps} \in \Inc$ as defined in Equation \eqref{eq:induced-by-eps} and by suitably restricting any $\pi \in \Inc$. By \cite[Theorem 3.6]{HS}, every such filtration stabilizes. Thus, Lemma~\ref{lem:gen-vs-stability} gives that $\Ib$ is generated in finitely many degrees. Since each ring $\XO{1}_n$ is noetherian, $\Ib$ has a finite generating set.

Using \cite[Corollary 3.7]{HS} instead of \cite[Theorem 3.6]{HS}, one gets that $\XI{1}$ is noetherian.

Let $d \ge 2$. It is enough to show the claim for $d = 2$. We adapt \cite[Example 3.8]{HS}) and use the identifications in Remark~\ref{rem:poly-alg}.
For $i \ge 3$, consider the monomial $u_i = x_{1,2} x_{2,3} \cdots x_{i-1,i} x_{1,i}$ representing a cycle on the vertices $1,2,\ldots,i$. Any morphism $\eps\colon S \hookrightarrow T$ maps such a monomial onto a monomial that presents a cycle of the same length. Since no cycle contains a strictly smaller cycle as a subgraph, it follows that none of the image monomials of $u_i$ divides an image of $u_j$ if $i < j$. Thus, we get a strictly increasing sequence of $\FIO$-ideals (or $\FI$-ideals, respectively)
\[
\la u_3 \ra \subsetneqq \la u_3, u_4 \ra \subsetneqq \cdots
\]
We conclude by Proposition \ref{prop:char-noeth-module}.
\end{proof}

We later show (see Theorem \ref{thm:finite-G-basis}) that $\XO{1}$ and $\XI{1}$ are noetherian if $K$ is an arbitrary noetherian ring.

\begin{rem}
Using the above arguments and the results in \cite{HS}, it also follows that the tensor products $(\XO{1})^{\otimes c}$ and $(\XI{1})^{\otimes c}$ are noetherian for every $c \in \N$. This motivates the question whether tensor products of noetherian $\FIO$-algebras and similar constructions always produce noetherian algebras.

\end{rem}


\section{Well-partial-orders}
\label{sec:orders}

We are now going to establish a combinatorial result that we will use when we discuss Gr\"obner bases for $\FIO$-modules. This section can be read independently of other parts of the paper.

Recall that a \emph{well-partial-order} on a set $S$ is a partial order $\le$ such that, for any infinite sequence $s_1,s_2, \ldots$ of elements in $S$, there is a pair of indices $i < j$ such that $s_i \le s_j$.

\begin{rem}
\label{rem:well-partial-order}
\
\begin{enumerate}
\item
If $S$ and $T$ are sets which have well-partial-orders, then it is well known that their Cartesian product $S \times T$ also admits a well-partial order, namely the componentwise partial order defined by $(s,t) \le (s', t')$ if $s \le s'$ and $t \le t'$. The analogous statement is true for finite products. In particular, it follows that the componentwise partial order on $\N_0^c$ is a well-partial-order, a result which is also called Dickson's Lemma.
\item
Given a set $S$ with a partial order $\le$, define a partial order on the set $S^*$ of finite sequences of elements in $S$ by
$(s_1,\ldots,s_p) \le_H (s'_1,\ldots,s'_q)$ if there is a strictly increasing map $\phi\colon [p] \to [q]$ such that $s_i \le s'_{\phi(i)}$ for all $i \in [p]$. This order is called the \emph{Higman order} on $S^*$. It is a well-partial-order by Higman's Lemma (see \cite{H} or, for example, \cite{Draisma}).
\end{enumerate}
\end{rem}

Before defining the relation we are interested in, recall the definition of the sign of an integer $n$:
\[
\sgn (n) = \begin{cases}
1 & \text{if } n > 0; \\
0 & \text{if } n = 0; \\
-1 & \text{if } n < 0.
\end{cases}
\]

\begin{defn}
Let $S$ be any set with a partial order $\le$. For any non-negative integer $d$, define a relation on
$(S \times \N)^{d+1}$ by
\[
(s_0, i_0) \times \cdots \times (s_d, i_d) \preceq (t_0, j_0) \times \cdots \times (t_d, j_d)
\]
if $(s_0,\ldots,s_d) \le (t_0,\ldots,t_d)$ in the componentwise partial order on $S^{d+1}$ and
$\sgn (i_k-i_0) = \sgn (j_k - j_0)$ for each $k = 1,\ldots,d$. (The second condition is empty if $d=0$.)
\end{defn}

This is a partial order on $(S \times \N)^{d+1}$. In fact, more is true.

\begin{prop}
\label{prop:Higman-on-products}
If $\le$ is a well-partial-order on $S$, then $\preceq$ is  a well-partial-order on $(S \times \N)^{d+1}$.

In particular, the induced Higman order $\preceq_H$ on the set $((S \times \N)^{d+1})^*$ is a well-partial-order for every $d \in \N_0$.
\end{prop}

\begin{proof}
The second part follows from the first one by Remark \ref{rem:well-partial-order}(ii). Thus, it is enough to show the first assertion.

To this end consider any infinite sequence $t_1, t_2, \ldots$ of elements in $(S \times \N)^{d+1}$, where $t_k = (s_{k, 0}, i_{k, 0}) \times \cdots \times (s_{k, d}, i_{k, d})$. Since the componentwise partial order on $S^{d+1}$ is a well-partial-order, it is well-known (see, e.g., \cite[page 298]{Kru-72}) that there is an infinite
sequence $n_1 < n_2 < \cdots$ of positive integers such that
\[
(s_{n_1, 0},\ldots,s_{n_1,d}) \le (s_{n_2, 0},\ldots,s_{n_2, d}) \le \cdots.
\]
Consider now the corresponding infinite subsequence of the original sequence of elements in $(S \times \N)^{d+1}$:
\[
(s_{n_1, 0}, i_{n_1, 0}) \times \cdots \times (s_{n_1, d}, i_{n_1, d}), \
(s_{n_2, 0}, i_{n_2, 0}) \times \cdots \times (s_{n_2, d}, i_{n_2, d}), \ \ldots.
\]
Since there are only $9$ different pairs $(\sgn(m), \sgn (n))$ for integers $m, n$, there must be positive integers $n_k < n_l$ such that
\[
(s_{n_k, 0}, i_{n_k, 0}) \times \cdots \times (s_{n_k, d}, i_{n_k, d}) \preceq
(s_{n_l, 0}, i_{n_l, 0}) \times \cdots \times (s_{n_l, d}, i_{n_l, d}),
\]
that is, $t_{n_k} \preceq t_{n_l}$. Hence $\preceq$ is a well-partial-order.
\end{proof}


\section{Gr\"obner Bases of $\FIO$-modules}
\label{sec:groebner}

Throughout this section we fix a positive integer $c$ and
 consider modules over $\Pb \cong (\XO{1})^{\otimes c}$ and $(\XI{1})^{\otimes c}$, respectively, over an arbitrary noetherian ring $K$.

First, we study free modules over $\Pb$. Recall that (see Example~\ref{exa:OI-alg}):
\[
\Pb_m = K[x_{i, j} \; : \; i \in [c], \ j \in [m]]
\]
for every $m \in \N_0$, and $\eps^* (x_{i, j}) = x_{i, \eps(j)}$ for $i \in [c], \ j \in [m]$ and  $\eps \in \Hom_{\FIO} ([m], [n])$.
Furthermore, for each integer $d \ge 0$, the free module $\Fo{d}$ over $\Pb$ is given by
\[
\Fo{d}_m = \oplus_{\pi} \Pb_m e_{\pi} \cong (\Pb_m)^{\binom{m}{d}},
\]
where the sum is taken over all $\pi \in \Hom_{\FIO} ([d], [m])$,
and
\[
\Fo{d}(\eps)\colon \Fo{d}_m \to \Fo{d}_n, \ a e_{\pi} \mapsto \eps^*(a) e_{\eps \circ \pi},
\]
where $a \in \Pb_m$ and $\eps \in \Hom_{\FIO} ([m], [n])$.

A \emph{monomial} in $\Fo{d}$ is an
element of some $\Fo{d}_m$ of the form
\[
x_{\lpnt, 1}^{u_1} \cdots x_{\lpnt, m}^{u_m}e_{\pi}, \quad \text{ where } \pi \in \Hom_{\FIO} ([d], [m]) \text{ for some } m, \; u_j \in \N_0^c,
\]
 the $i$-th entry of $u_j \in \N_0^c$ is the exponent of the variable $x_{i, j}$, and $x_{\lpnt, j}^{u_j}$ is the product of these powers. We want to show that the monomials in $\Fo{d}$ admit a well-partial-order.

We consider a divisibility relation that is compatible with the $\FIO$-module structure.

\begin{defn}
\label{def:FIO-div}
A monomial $\nu = x_{\lpnt, 1}^{v_1} \cdots x_{\lpnt, n}^{v_n}e_{\rho} \in \Fo{d}_n$ is said to be \emph{$\FIO$-divisible} by
a monomial $\mu = x_{\lpnt, 1}^{u_1} \cdots x_{\lpnt, m}^{u_m}e_{\pi} \in \Fo{d}_m$ if there is an $\eps \in \Hom_{\FIO} ([m], [n])$ such that $\Fo{d} (\eps) (\mu)$ divides $\nu$ in $\Fo{d}_n$, that is,
$\eps^* (x_{\lpnt, 1}^{u_1} \cdots x_{\lpnt, m}^{u_m})$ divides $x_{\lpnt, 1}^{v_1} \cdots x_{\lpnt, n}^{v_n}$ in $\Pb_n$ and
$\rho = \eps \circ \pi$. In this case we write $\mu \big |_{\FIO} \nu$.
\end{defn}

It is worth writing out this definition more explicitly. We get
\begin{equation}
\label{eq:def-divisible}
\begin{split}
x_{\lpnt, 1}^{u_1} \cdots x_{\lpnt, m}^{u_m}e_{\pi} \big |_{\FIO} x_{\lpnt, 1}^{v_1} \cdots x_{\lpnt, n}^{v_n}e_{\rho} & \text{ if and only if there is some } \eps \in \Hom_{\FIO} ([m], [n]) \\
& \; \text{such that } \rho = \eps \circ \pi \text{ and } u_i \le v_{\eps (i)} \text{ for each } i \in [m].
\end{split}
\end{equation}

$\FIO$-divisibility is certainly a partial order. In fact, as an application of the results in the previous section, we show that more is true.

\begin{prop}
\label{prop:monomials-in-Fd}
For each $d \in \N_0$, $\FIO$-divisibility is a well-partial-order on the set of monomials in $\Fo{d}$.
\end{prop}

\begin{proof}
We want to apply Proposition \ref{prop:Higman-on-products} to $S = \N_0^{c}$. To this end we encode a
monomial $\mu = x_{\lpnt, 1}^{u_1} \cdots x_{\lpnt, m}^{u_m}e_{\pi}$ as a sequence of $m$ elements
in $(\N_0^{c} \times \N) \times (\N_0^c \times \N)^d$, namely
\[
s(\mu) = \big ((u_1, 1) \times s' (\mu), (u_2, 2) \times s' (\mu),\ldots, (u_m, m) \times s' (\mu) \big ) \in ((\N_0^c \times \N)^{d+1})^*,
\]
where
\[
s' (\mu) = (u_{\pi (1)}, \pi (1)) \times \cdots \times (u_{\pi (d)}, \pi (d)) \in (\N_0^c \times \N)^d.
\]

Consider any infinite sequence $\mu_1, \mu_2, \ldots $ of monomials in $\Fo{d}$. Proposition \ref{prop:Higman-on-products} shows that there are indices $p < q$ such that the encoded monomials satisfy $s( \mu_p) \preceq_H s (\mu_q)$. To simplify notation write
\[
\mu_p = x_{\lpnt, 1}^{u_1} \cdots x_{\lpnt, m}^{u_m}e_{\pi} \; \text{ and } \; \mu_q = x_{\lpnt, 1}^{v_1} \cdots x_{\lpnt, n}^{v_n}e_{\rho}.
\]
The relation $s( \mu_p) \preceq_H s (\mu_q)$ means that there is a strictly
increasing map $\eps\colon [m] \to [n]$ such that, for every $i \in [m]$, one has
\[
(u_i,u_{\pi (1)},\ldots,u_{\pi (d)}) \le (v_{\eps (i)}, v_{\rho (1)},\ldots,v_{\rho (d)})
\]
and
\[
\sgn (\pi (j) - i) = \sgn (\rho (j) - \eps (i)) \text{ for } j \in [d].
\]
In particular, this gives $u_i \le v_{\eps(i)}$ for every
$i \in [m]$. Furthermore, if $i = \pi (j)$, then we get $\rho (j) = \eps (i) = \eps (\pi (j))$, which proves $\rho = \eps \circ \pi$. Comparing with \eqref{eq:def-divisible}, we conclude that
$\mu_p$ $\FIO$-divides $\mu_q$, which completes the argument.
\end{proof}

Now we want to develop Gr\"obner bases theory for submodules of free modules over $\Pb$. We adapt some ideas in \cite{AH}. Fix some integer $d \ge 0$, and consider the free $\FIO$-module $\Fo{d}$. We need a suitable order on the monomials.

\begin{defn}
A \emph{monomial order} on $\Fo{d}$ is a total order $>$ on the monomials of $\Fo{d}$ such that
 if $\mu, \nu$ are monomials in $\Fo{d}_m$ , then $\mu > \nu$ implies:
\begin{enumerate}
\item $u \mu > u \nu > \nu$ for every monomial $u \neq 1$ in $\Pb_m$;

\item $\Fo{d} (\eps) (\mu) > \Fo{d} (\eps) (\nu)$ for every $\eps \in \Hom_{\FIO} ([m], [n])$; \quad and

\item $\Fo{d} (\iota_{m, n}) (\mu) > \mu$ whenever $n > m$.
\end{enumerate}

\end{defn}

Such orders exist.

\begin{ex}
Order the monomials in every polynomial ring $\Pb_m$ lexicographically with $x_{i, j} > x_{i', j'}$ if either $i < i'$ or $i = i'$ and $j < j'$. Identify a monomial $e_{\pi} \in \Fo{d}_m$ with a vector $(m, \pi (1),\ldots,\pi (d)) \in \N^{d+1}$ and order such monomials by using the lexicographic order on $\N^{d+1}$. For example, this implies that every $e_{\pi} \in \Fo{d}_m$ is smaller than any $e_{\tilde{\pi}} \in \Fo{d}_n$ if $m < n$.

Finally, for monomials $u e_{\pi}$ and $v e_{\tilde{\pi}}$, define $u e_{\pi} > v e_{\tilde{\pi}}$ if either $e_{\pi} > e_{\tilde{\pi}}$ or $e_{\pi} = e_{\tilde{\pi}}$ and $u > v$ in $\Pb_m$, where $e_{\pi}, e_{\tilde{\pi}} \in \Fo{d}_m$. One checks that this gives indeed a monomial order on $\Fo{d}$.
\end{ex}

Observe that every monomial order on $\Fo{d}$ refines the partial order defined by $\FIO$-divisibility. Thus,
Proposition \ref{prop:monomials-in-Fd} has the following consequence.

\begin{cor}
\label{cor:unique-min-element}
Fix any monomial order $>$ on $\Fo{d}$. Every non-empty set of monomials of $\Fo{d}$ has a unique minimal element in the order $>$.
\end{cor}

\begin{proof}
By Proposition \ref{prop:monomials-in-Fd}, any set of monomials $\Mcc \neq \emptyset$ of $\Fo{d}$ has finitely many distinct minimal elements with respect to $\FIO$-divisibility, say $\mu_1,\ldots,\mu_s$ (see \cite{Kru}). Assume that $\mu_1$ is the smallest of these $s$ monomials in the order $>$. We claim that $\mu_1$ is the desired smallest element of $\Mcc$. Indeed, if $\nu$ is any monomial in $\Mcc$, then, by the choice of $\mu_1,\ldots,\mu_s$, there is some $k \in [s]$ such that $v$ is $\FIO$-divisible by $\mu_k$. Thus, the properties of a monomial order imply $\nu \ge \mu_k$. We conclude by noting that $\mu_k > \mu_1$.
\end{proof}

If $\M$ is any $\FIO$-module we often write instead $q \in \coprod_{0\leq n} \M_n$ simply $q \in \M$ and refer to $q$ as an \emph{element} of $\M$. For example, this leads to the notion of a \emph{subset} of $\M$.

\begin{defn}
Let $>$ be a monomial order on $\Fo{d}$.
Consider an element $q = \sum c_\mu \mu \in \Fo{d}_m$ for some $m \in \N_0$ with monomials $\mu$ and coefficients
$c_{\mu} \in K$. If $q \neq 0$ we define its \emph{leading
monomial} $\lm (q)$ as the largest monomial $\mu$ with a non-zero coefficient $c_{\mu}$. This coefficient is called the \emph{leading coefficient}, denoted $\lc (q)$. The \emph{leading term} of $q$ is $\lt (q) = \lc (q) \cdot \lm (q)$.

 If $q$ ranges over the elements of a subset $E$ of $\Fo{d}$, we use $\lm (E), \lc (E), \lt (E)$ to denote the sets of the corresponding elements.
\end{defn}

For a subset $E$ of any $\FIO$-module $\M$, it is convenient to denote by
$\la E \ra_{\M}$ the smallest $\FIO$-submodule of $\M$ that contains $E$. It is called the
$\FIO$-submodule \emph{generated by $E$}.

\begin{rem}
\label{rem:gen-vs-divisibility}
Let $\mu, \nu$ be two monomials of $\Fo{d}$. Then $\mu \big |_{\FIO}\, \nu$ if and only if $\nu \in \la \mu \ra_{\Fo{d}}$.
\end{rem}

\begin{defn}
Fix a monomial order $>$ on $\Fo{d}$,
and let $\M$ be an $\FIO$-submodule of $\Fo{d}$.
\begin{enumerate}
\item The \emph{initial module} of $M$ is
\[
\ini (M) = \la \lt (q) \s q \in M \ra_{\Fo{d}}.
\]
It is a submodule of $\Fo{d}$.

\item A subset $B$ of $\M$ of $F(d)$ is a \emph{Gr\"obner basis} of $M$ (with respect to $>$) if
\[
\ini (M) = \la \lt (B) \ra_{\Fo{d}}.
\]

\end{enumerate}
\end{defn}

We do not require that a Gr\"obner basis is finite although finite Gr\"obner bases are of course more useful. Our goal is to show that the latter exist.

\begin{defn}
Let $B$ be any subset of $\Fo{d}$. We say that an element $q \in \Fo{d}_m$ \emph{reduces to $r \in \Fo{d}_m$ by $B$} if there is some $q' \in \la B \ra_{\Fo{d}} \cap \Fo{d}_m$ such that
\[
q = q' + r \quad \text{ and either } \quad \lm (r) < \lm (q) \text{ or } r = 0.
\]
In this case, it is said that $q$ is \emph{reducible by $B$}.
\end{defn}

\begin{rem}
If $K$ is a field, then $q$ is reducible by $B$ if and only if there is some $b \in B$ such that $\lm (b) \big |_{\FIO}\, \lm (q)$ (see Remark \ref{rem:gen-vs-divisibility}).
\end{rem}

Iterating reductions gives a division algorithm.

\begin{defn}
Given an element $q \in \Fo{d}$, an element $r \in \Fo{d}$ is said to be a \emph{remainder of $q$ on dividing by $B$} or a \emph{normal form of $q$ modulo $B$} if there is a sequence of elements $q_0, q_1,\ldots,q_s \in \Fo{d}$ such that $q = q_0$, $r = q_s$, each $q_{i+1}$ is a reduction of $q_i$ by $B$, and either $r = 0$ or $r \neq 0$ is not reducible by $B$.
\end{defn}

As in the classical noetherian setting, one has the following equivalence.

\begin{prop}
\label{prop:G-basis-criterion}
Let $\M$ be any $\FIO$-submodule of $\Fo{d}$, and let $B$ be a subset of $\Fo{d}$.
Then the following conditions are equivalent:
\begin{enumerate}
\item $B$ is a Gr\"obner basis of $\M$;

\item every $q \neq 0$ in $\M$ is reducible modulo $B$;

\item every $q \in \M$ has remainder zero modulo $B$.
\end{enumerate}

In particular, any Gr\"obner basis of $\M$ generates $\M$.
\end{prop}

\begin{proof} The definitions give that (i) implies (ii) and that (i) is a consequence of (iii). Moreover, (iii) yields the final assertion.

Assume (ii) is true, and consider any $q_0 \neq 0$ in $\M$. Reducing $q_0$ by $B$ we get some $q_1 \in \M$ with $\lm (q_1) < \lm (q_0)$ or $q_1 = 0$. If $q_1 = 0$ we are done. Otherwise we reduce $q_1$. Repeating if necessary, this process terminates because of Corollary~\ref{cor:unique-min-element}. This shows (iii).
\end{proof}

If $K$ is a field, then a subset $B$ of $\Fo{d}$ is a Gr\"obner basis of $\M$ if and only if $\la \lm (B) \ra_{\Fo{d}} = \la \lm (\M) \ra_{\Fo{d}}$. For the general case, we need one more step.

\begin{defn}
Given an $\FIO$-submodule $\M$ of $\Fo{d}$ and a monomial $\mu \in \Fo{d}_m$, consider the set
\[
\lc (\M, \mu) = \{\lc (q) \s q \in \M_m \text{ and } \lm (q) = \mu\} \cup \{0_K\}.
\]
It is an ideal of $K$.
\end{defn}

Note that if $\mu, \nu$ are monomials of $\Fo{d}$ with $\mu \big |_{\FIO}\, \nu$, then $\lc (\M, \mu) \subset \lc (\M, \nu)$.

Recall our standing assumption in this section that $K$ is a noetherian ring.

\begin{thm}
\label{thm:finite-G-basis}
Fix any integer $d \ge 0$ and a monomial order $>$ on $\Fo{d}$.
Every $\FIO$-submodule of the free $\FIO$-module $\Fo{d}$ over $\Pb \cong (\XO{1})^{\otimes c}$ has a finite Gr\"obner basis (with respect to $>$).
\end{thm}

\begin{proof}
Let $\M \neq 0$ be any $\FIO$-submodule of $\Fo{d}$.
Define a new partial order on the set of monomials of $\Fo{d}$ by
\[
\mu \le_{\M} \nu \quad \text{ if } \mu \big |_{\FIO}\, \nu \text{ and } \lc (\M, \mu) =\lc (\M, \nu).
\]
We claim that this is a well-partial-order on the set of monomials of $\Fo{d}$.

Indeed, consider any infinite sequence $\mu_1, \mu_2, \ldots $ of monomials in $\Fo{d}$. Since
$\FIO$-divisibility is a well-partial-order by
Proposition \ref{prop:monomials-in-Fd}, passing to a suitable infinite subsequence if necessary, we
 may assume that $\mu_i \big |_{\FIO}\, \mu_{i+1}$ for all $i$. Hence, we
get an ascending sequence of ideals of $K$
\[
\lc (\M, \mu_1) \subset \lc (\M, \mu_2) \subset \cdots.
\]
Since $K$ is noetherian, this sequence stabilizes. Thus, there is some index $i$ with
$\mu_i \le_{\M} \mu_{i+1}$, which shows that $\le_{\M}$ is a well-partial-order.

It follows (see \cite{Kru}) that
the set $\lm (\M)$ of leading monomials has finitely many minimal monomials in the order $\le_{\M}$,
say,
$\mu_1,\ldots,\mu_t$. Thus, for every monomial $\mu \in \lm (\M)$, there is some
$\mu_i$ such that $\mu_i \big |_{\FIO}\, \mu$ and $\lc (\M, \mu_i) =\lc (\M, \mu)$.
Since every ideal $\lc (\M, \mu_i)$ is finitely generated, there is a finite subset $E_i \subset \M$ such
 that $\lc (E_i)$ generates the ideal
$\lc (\M, \mu_i)$. We claim that $B = E_1 \cup \ldots \cup E_t$ is a Gr\"obner basis of $\M$.

Indeed, consider any element $q \neq 0$ of $\M_m$ for some $m$. Put $\mu = \lm (q)$. Choose
$i \in [t]$ such that $\mu_i \big |_{\FIO}\, \mu$ and $\lc (\M, \mu_i) =\lc (\M, \mu)$. Hence, there are
elements $b_1,\ldots,b_s \in E_i \subset B$ such that $\lc (q) = k_1 \lc (b_1) + \cdots k_s \lc (b_s)$
for suitable $k_1,\ldots,k_s \in K$. Suppose $\mu_i$ is in $\M_{m_i}$. Then $\mu_i \big |_{\FIO}\, \mu$
gives that there are a morphism $\eps \in \Hom_{\FIO} (m_i, m)$ and a monomial $\nu \in \Pb_m$ such that $\mu = \nu \M (\eps) (\mu_i)$ (see Definition \ref{def:FIO-div}). It follows that the leading term of $k_1 \nu \M (\eps) (b_1) + \cdots k_s \nu \M (\eps) (b_s)$ is $\lc (q) \mu = \lt (q)$. This shows that $q$ is reducible by $B$. Thus, $B$ is a finite Gr\"obner basis of $\M$ by Proposition \ref{prop:G-basis-criterion}.
\end{proof}

The main result of this section follows now quickly.

\begin{thm}
\label{thm:fg-gives-noeth-mod}
\
\begin{enumerate}
\item Every finitely generated $\FIO$-module over $\Pb \cong (\XO{1})^{\otimes c}$ is noetherian.

\item Every finitely generated $\FI$-module over $(\XI{1})^{\otimes c}$ is noetherian.
\end{enumerate}
 \end{thm}

\begin{proof}
(i) Let $\M$ be a finitely generated $\FIO$-module over $\Pb$.
Combining Theorem~\ref{thm:finite-G-basis} and Proposition~\ref{prop:G-basis-criterion}, it follows that every free $\FIO$-module $\Fo{d}$ over $\Pb$ is noetherian. Hence
Propositions~\ref{prop:finite-gen} and \ref{prop:noeth-along-sequences} imply that $\M$ is a quotient of a noetherian $\FIO$-module, and so $\M$ also is noetherian.

(ii) As observed above, every submodule of a free $\FI$-module $\Fi{d}$ over $(\XI{1})^{\otimes c}$ may also be considered as an $\FIO$-module over $\Pb$. Moreover, $\Fi{d}$ is finitely generated as $\FIO$-module by Remark~\ref{rem:free}(iv), and so it is noetherian as an $\FIO$-module by (i). It follows that it also is noetherian as an $\FI$-module. Now we conclude as in (i).
\end{proof}

After completing the first version of this paper Jan Draisma kindly informed as that Theorem \ref{thm:fg-gives-noeth-mod}(ii) was independently shown in \cite{DKK}.

As a first consequence of Theorem \ref{thm:fg-gives-noeth-mod} we recover main results of \cite{CEFN, CEF}
(see \cite[Theorem A]{CEFN} and \cite[Theorem 1.3]{CEF})
in the case of $\FI$-modules and,
for $\FIO$-modules, of \cite{SS-14}
(see \cite[Theorem 7.1.2]{SS-14}).

\begin{cor}
\label{cor:church}
\
\begin{enumerate}
\item Every finitely generated $\FIO$-module over $\XO{0}$ is noetherian.

\item Every finitely generated $\FI$-module over $\XI{0}$ is noetherian.
\end{enumerate}
\end{cor}

\begin{proof}
Since $\XO{d}$ is a quotient of $\XO{1}$ as an $\FIO$-algebra and the analogous statement is true for $\XI{d}$, the claims follow by Theorem~\ref{thm:fg-gives-noeth-mod}.

Alternately, (i) is a consequence of the arguments for Theorem~\ref{thm:finite-G-basis} by noting that every $\Fo{d}_m$ is generated as an $\XO{0}_m$ module by a single monomial. Now (ii) follows from (i) as in Theorem~\ref{thm:fg-gives-noeth-mod}.
\end{proof}

The arguments in Theorem \ref{thm:finite-G-basis} can be somewhat extended. This gives an $\FIO$-analog of a folklore result about monomial subalgebra of polynomial rings in finitely many variables.

\begin{prop}
\label{prop:noeth-subalgebra}
Let $\Bb$ be a subalgebra of $\Pb \cong (\XO{1})^{\otimes c}$ that is generated by monomials over $K$. Assume that for any monomials $\mu, \nu \in \Bb$ the fact that $\mu = \eps^* (\nu) \cdot \kappa$ (in some $\Pb_n$) implies $\kappa \in \Bb$.
Then every $\FIO$-submodule of a free $\FIO$-module $\Fo{d}$ ($d \in \N_0$) over $\Bb$ has a finite Gr\"obner basis (with respect to any monomial order on $\Fo{d}$).

In particular, every finitely generated $\FIO$-module over $\Bb$ is noetherian.
\end{prop}

\begin{proof}
As above, the second assertion is a consequence of the first one and Proposition~\ref{prop:G-basis-criterion}. To prove the first claim fix any $d \in \N_0$.
The assumption guarantees that $\FIO$-divisibility of monomials in $\Fo{d}$ over $\Pb$ remains a
well-partial-order when restricted to the set of monomials in $\Fo{d}$ over $\Bb$. Hence, the proof of Theorem~\ref{thm:finite-G-basis} shows that every submodule of $\Fo{d}$ has a finite Gr\"obner basis.
\end{proof}

For an $\FI$-module this implies as in Theorem~\ref{thm:fg-gives-noeth-mod}:

\begin{cor}
\label{cor:noeth-subalgebra}
Let $\Bb$ be a subalgebra of $(\XI{1})^{\otimes c}$ that is generated by monomials over $K$. Assume that for any monomials $\mu, \nu \in \Bb$ the fact $\mu = \eps^* (\nu) \cdot \kappa$
(in some $\bigl((\XI{1})^{\otimes c} \bigr)_n$) implies $\kappa \in \Bb$. Then every finitely generated $\FI$-module over $\Bb$ is noetherian.
\end{cor}

Recall that Veronese subalgebras were introduced in Example~\ref{exa:Veronese}.

\begin{ex}
\begin{enumerate}
\item
The assumption in Proposition~\ref{prop:noeth-subalgebra} and Corollary~\ref{cor:noeth-subalgebra}, respectively, is satisfied for the Veronese subalgebras of $(\XO{1})^{\otimes c}$ and $(\XI{1})^{\otimes c}$. Hence, for every $e \in \N$, finitely generated $\FIO$-modules over $((\XO{1})^{\otimes c})^{(e)}$ and finitely generated $\FI$-modules over $((\XI{1})^{\otimes c})^{(e)}$ are noetherian.
\item
Consider the following subset of monomials of $(\XI{1})^{\otimes c}$
\[
Y =\{ x_{1, j_1} \cdots x_{c, j_c} \; | \; j_1,\ldots,j_c \in \N\}.
\]
We claim that the subalgebra $\Bb$ of $(\XI{1})^{\otimes c}$ that is generated by $Y$
 satisfies the assumption of Corollary~\ref{cor:noeth-subalgebra}. To see this, write a monomial in $(\XI{1})^{\otimes c}$ as
\[
\mu = x_{1, \lpnt}^{u_1} \cdots x_{c, \lpnt}^{u_c} \quad \text{ with } \; x_{i,\lpnt}^{u_i} = \prod_{j \in \N} x_{i, j}^{u_{i, j}},
\]
where $u_i = (u_{i, 1}, u_{i, 2},\ldots)$ is a sequence of non-negative integers with only finitely many positive entries. Set $|u_i| = \sum_{j \in \N} u_{i, j}$. Then the monomial $\mu$ is in $\Bb$ if and only if $|u_1| = |u_2| = \cdots = |u_c|$. Now the claimed divisibility condition follows. Hence, Corollary~\ref{cor:noeth-subalgebra} shows that every finitely generated $\FI$-module over $\Bb$ is noetherian. Note that, for $n \in \N$, the algebra $\Bb_{n+1}$ is the coordinate ring of the $c$-fold Segre product $\PP^n_K \times \cdots \times \PP^n_K$ of projective spaces over
$K$.
\end{enumerate}
\end{ex}

The above results motivate the following:

\begin{conj}
Every finitely generated $\FIO$-module (or $\FI$-module, respectively) over a noetherian $\FIO$-algebra (or $\FI$-algebra, respectively) is noetherian.
\end{conj}

Recall that $K[X] \cong \lim \Pb_n$. Hence, Corollary~\ref{cor:noeth-limit} and Theorem~\ref{thm:fg-gives-noeth-mod} imply:

\begin{cor}
\label{cor:limits-alg-noeth}
Let $K$ be a noetherian ring. For every $c \in \N$, one has:
\begin{enumerate}
\item $K[X]$ is an $\Inc$-noetherian $K$-algebra.

\item $K[X]$ is an $\Sym (\infty)$-noetherian $K$-algebra.
\end{enumerate}
\end{cor}

Note that this result extends \cite[Theorems 1.1 and 3.1]{HS} from coefficients in a field $K$ to arbitrary noetherian rings. Part (ii) was shown in the special case $c=1$ in \cite[Theorem 1.1]{AH}.


\section{Stabilization of Syzygies}
\label{sec:syzygies}

The main goal of this section is to study homological properties of
a finitely generated $\FIO$-module (or $\FI$-module, respectively) $\M$
in the cases that the underlying category of modules is noetherian, that is,
every finitely generated module is noetherian.
See Theorem \ref{thm:fg-gives-noeth-mod} or Corollary \ref{cor:church} for such situations.

Moreover, these homological properties will be compared with the corresponding ones of the modules in the family $(\M_m)_{m\geq 0}$. In the following we fix a noetherian $\FIO$-algebra (or a noetherian $\FI$-algebra, respectively) $\Ab$ over a noetherian commutative ring $K$ such that the category of $\FIO$-modules (or $\FI$-modules, respectively) is noetherian.

A classical research topic in commutative algebra is the study of syzygies and induced invariants such as Betti numbers or the Castelnuovo-Mumford regularity. Related results yield applications
to interesting problems in commutative as well as to other fields such as algebraic geometry or representation theory.

The framework developed in this manuscript yields an approach to study syzygies for the family of modules $(\M_m)_{m\geq 0}$ to which the next result can be applied.

\begin{thm}
\label{thm:stabilization-syz}
Let $\M$ be a finitely generated $\FIO$-module (or $\FI$-module, respectively) over $\Ab$. There exists a projective resolution $\Fb_{\lpnt}$ of $\M$ such that $\Fb_{p}$ is finitely generated for every $p\in \N_0$.
\end{thm}
\begin{proof}
By Proposition \ref{prop:finite-gen}, there is a presentation of $\M$ using a finitely generated free $\FIO$-module (or $\FI$-module, respectively) which is in particular a projective module. Since the category of $\FIO$-modules (or $\FI$-modules, respectively) under consideration is noetherian, the kernel of that presentation is again finitely generated. Thus we can construct inductively a projective resolution $\Fb_{\lpnt}$ of $\M$ such that $\Fb_{p}$ is finitely generated for every $p\in \N_0$. This concludes the proof.
\end{proof}

\begin{rem}
\
\begin{enumerate}
\item
As in \cite[Theorem A]{S} one could informally say, that given any integer $p \ge 0$, there is a finite list of \emph{master $p$-syzygies} of $\M$ from which all $p$-syzygies of $\M_m$ (defined by $\Fb_{\lpnt}$) can be obtained by combinatorial substitution procedures induced from the $\FIO$- or $\FI$-structure,
describing the $p$-th syzygies of $\M_m$ in terms of ``earlier'' ones. Indeed, theses master syzygies are the finitely many generators of $\Fb_{p}$.

\item
If $\M$ is a $\Z$-graded module over a standard $\Z$-graded algebra $\Ab$ (see Definition~\ref{def:ca-cases} below), then the constructed resolution $\Fb_{\lpnt}$ can be chosen to consist only of graded modules with all maps being homogeneous of degree 0. We refer to this sequence as  a \emph{graded resolution} of $\M$.

\item
Theorem \ref{thm:stabilization-syz} and its applications
may be seen as a general framework to get results
as pioneered in \cite[Theorem A]{S}. Theorem \ref{thm:stabilization-syz} above does not cover
\cite[Theorem A]{S}, because the latter result applies to some modules over a non-noetherian $\FI$-algebra.
Instead, Theorem \ref{thm:stabilization-syz} applies to \emph{all} finitely generated modules over a
noetherian $\FI$-algebra.
\end{enumerate}
\end{rem}

\begin{ex}
\label{exa:det}
For a simple specific instance illustrating the above result, fix an integer $c \ge 1$ and consider generic
$c \times n$ matrices $X_n$ whose entries are the variables $x_{i, j}$ with $i \in [c]$ and $j \in [n]$. Fix any integer
$t$ with $1 \le t \le c$, and denote by $I_n$ the ideal generated by the $t \times t$ minors of $X_n$. The
sequence $(I_n)_{n \in \N_0}$ determines an ideal of $(\XI{1})^{\otimes c}$. Hence,
Theorem \ref{thm:stabilization-syz} shows that for every integer $p \ge 0$, there is an integer $n_p$ such that, for
every $n \ge n_p$, the $p$-th syzygies of $I_n$ over $(\XI{1}_n)^{\otimes c}$ can be obtained from the $p$-th syzygies
of $I_{n_p}$ over $(\XI{1}_{n_p})^{\otimes c}$.

Notice that this is true regardless of the characteristic even though it is known that resolutions of determinantal ideals behave differently in positive characteristic (see \cite{Ha}).
\end{ex}

In the following, if we consider a projective resolution $\Fb_{\lpnt}$ of a finitely generated module $\M$,
then we always construct $\Fb_{\lpnt}$ using Proposition \ref{prop:finite-gen} as in the proof of Theorem \ref{thm:stabilization-syz}. In particular, all modules in such a resolution are free in the sense of Definition \ref{def:free}.

The next goal is to study Betti numbers. For this we have to introduce some notations and discuss some preparations before presenting our main results.

Given two $\FIO$-modules (or $\FI$-modules, respectively) $\M$ and $\Nb$, their \emph{tensor product}
is the $\FIO$-module (or $\FI$-module, respectively)
$\M \tensor_\Ab \Nb$, defined componentwise as
$$
\M_S \tensor_{\Ab_S} \Nb_S
$$
for every totally ordered finite set $S$ (or finite set $S$, respectively)
with obvious morphisms induced by the ones given by $\M$ and $\Nb$.
Since taking the tensor product is right exact, this yields $\FIO$-modules (or $\FI$-modules, respectively)
$$
\Tor_p^\Ab(\M, \Nb) \text{ for } p\geq 0
$$
defined as left-derived functors. The first crucial observation is:
\begin{lem}
\label{lem:fg-tor}
Let $\M,\Nb$ be finitely generated $\FIO$-modules (or $\FI$-modules, respectively) over $\Ab$
and let $p\in \N_0$. Then $\Tor_p^\Ab(\M, \Nb)$ is finitely generated. Moreover, we have
$$
\Tor_p^\Ab(\M, \Nb)_S=\Tor_p^{\Ab_S}(\M_S, \Nb_S)
$$
for every totally ordered finite set $S$ (or finite set $S$, respectively).
\end{lem}
\begin{proof}
By Theorem \ref{thm:stabilization-syz} there exists a projective resolutions $\Fb_{\lpnt}$ of $\M$ such that $\Fb_{p}$ is finitely generated for every $p\in \N_0$. Since $\Tor_p^\Ab(\M, \Nb)$ is the homology of $\Fb_{\lpnt} \tensor_{\Ab} \Nb$, a complex of finitely generated modules, and the underlying module category is noetherian, it follows that $\Tor_p^\Ab(\M, \Nb)$ is finitely generated.

Observe, that we have chosen $\Fb_{\lpnt}$ in such a way that $\Fb_{\lpnt,S}$ is a projective resolution of $\M_S$ over $\Ab_S$ using only finitely generated free modules. The additional claim is then a consequence of an isomorphism of $\Ab_S$-modules
$$
H_p(\Fb_{\lpnt} \tensor_{\Ab} \Nb)_S=H_p(\Fb_{\lpnt,S} \tensor_{\Ab_S} \Nb_S)
$$
for every totally ordered finite set $S$ (or finite set $S$, respectively).
\end{proof}

Next we define:
\begin{defn}
\label{def:ca-cases}
Let $\Ab$ be a noetherian $\FIO$-algebra (or $\FI$-algebra, respectively) over $K$.
\begin{enumerate}
\item
We call $\Ab$ a \emph{local}, if every $\Ab_S$ is a commutative local ring $(\Ab_S,\mm_S)$ and all morphisms $\Ab(\epsilon)$ are local homomorphisms. Here we always take $K=\Z$.
\item
Let $K$ be a field.
We say $\Ab$ is \emph{standard graded}, if every $\Ab_S$ is a finitely generated standard $\Z$-graded $K$-algebra, that is, $\Ab_S=\bigoplus_{d\geq 0} \Ab_{S,d}$ is a commutative finitely generated $\Z$-graded $K$-algebra generated in degree 1 with $\Ab_{S,0}=K$,
and all morphisms $\Ab(\epsilon)$ are homogeneous of degree 0.

We write $\mm_S=\bigoplus_{d\geq 1} \Ab_{S,d}$ for the uniquely determined graded maximal ideal of $\Ab_S$.
\end{enumerate}
\end{defn}

From now on we assume that $\Ab$ is either local or standard graded.
In the local case let $\Kb$ be the $\FIO$-module (or $\FI$-module, respectively)
defined by setting $\Kb_S=\Ab_S/\mm_S$ together with the natural morphisms. In the standard graded case we set similarly $\Kb_S=\Ab_S/\mm_S \cong K$.

\begin{defn}
\
\begin{enumerate}
\item
Let $\Ab$ be a local $\FIO$-algebra (or $\FI$-algebra, respectively),
and let $\M$ be a finitely generated $\FIO$-module (or $\FI$-module, respectively) over $\Ab$.
For every $p\in \N_0$ and every totally ordered finite set (or finite set, respectively) $S$ the \emph{Betti numbers} of $\M$
are defined as
$$
\beta_{S,p}^{\Ab}(\M)=\dim_{\Ab_S/\mm_S} \Tor_p^\Ab(\M, \Kb)_S.
$$
\item
Let $\Ab$ be a standard graded $\FIO$-algebra (or $\FI$-algebra, respectively),
and let $\M$ be a finitely generated graded $\FIO$-module (or $\FI$-module, respectively).
For every $p\in \N_0, j\in \Z$ and every totally ordered finite set (or finite set, respectively) $S$
we call
$$
\beta_{S,p,j}^{\Ab}(\M)=\dim_{K} \bigl(\Tor_p^\Ab(\M, \Kb)_S\bigr)_j
$$
the \emph{graded Betti number} of $\M$ with respect to $(S,p,j)$. Here we also set
$\beta_{S,p}^{\Ab}(\M)=\sum_j \beta_{S,p,j}^{\Ab}(\M)$
for the \emph{total Betti number} respect to $(S,p,j)$.
\end{enumerate}
\end{defn}

Our main result concerning stabilization of Betti numbers has its strongest form in the graded context:

\begin{thm}
\label{thm:stabilization}
Let $\Ab$ be standard graded $\FIO$-algebra (or $\FI$-algebra, respectively)
and let $\M$ be a finitely generated graded $\FIO$-module (or $\FI$-module, respectively).
Then for any $p\in \N$, one has
$$
m (\M, p)=
\max\{ j\in \Z:
\beta_{S,p,j}^{\Ab}(\M)\neq 0 \text{ for some } S
\}<\infty.
$$
Moreover, there are integers $j_0 (\M, p) < \cdots < j_t (\M, p) \le m (\M, p)$
%
such that for any $S$ with $|S|\gg 0$ we have
$$
\beta_{S,p,j}^{\Ab}(\M)
\neq 0
\text{ if and only if } j\in \{j_0 (\M, p),\dots,j_t (\M, p)\}.
$$
\end{thm}

\begin{proof}
By Lemma \ref{lem:fg-tor}, $\Nb=\Tor_p^\Ab(\M, \Kb)$ is
a finitely generated graded $\FIO$-module (or $\FI$-module, respectively).
In this situation one can choose a finite system
of generators $G \subset \coprod_{n\geq 0} \Nb_n$ which
is homogeneous with respect to the (internal) gradings of the modules $\Nb_n$.

As noted in Remark \ref{rem:graded-mod}(i), the maps $\Nb(\epsilon)$
are homogeneous of degree zero. As a consequence the degrees of
the elements in the induced homogenous systems of generators of $\Nb_S$
are the same for every totally ordered finite set (or finite set, respectively) $S$.
Some generators may only be needed for $|S|$ small
and become redundant for $|S|\gg 0$. But for $|S|\gg 0$ the degree sequence of
a minimal homogenous systems of generators of $\Nb_S$ has to stabilize because $\Nb$ has a finite generating set.

Indeed, this argument holds for every
finitely generated graded $\FIO$-module (or $\FI$-module, respectively).
In our case, $\Nb_S$ is a finitely dimensional $K$-vector space
and the degrees of minimal generators correspond to non-zero Betti numbers $\beta_{S,i,j}^{\Ab}(\M)$.
All statements of the theorem follow now immediately.
\end{proof}

\begin{rem}
\label{rem:stabilizationetc}
\

\begin{enumerate}
\item
There exists also a local version of Theorem~\ref{thm:stabilization}
in the sense that for $|S|\gg 0$
the $p$-th syzygy modules of $\M_S$ stabilizes, 
i.e.~there is a combinatorial pattern induced from the $\FIO$- or $\FI$-structure
describing the $p$-th syzygies in terms of ``earlier'' ones.

\item
Results as in Theorem~\ref{thm:stabilization} were previously known for Segre products (see \cite[Theorem A]{S} and \cite[Theorem 9.3.2]{SS-14}), and for secant varieties of Veronese varieties (see \cite[Theorem 1.1]{S-16} and \cite[Theorems 3.3 and 3.6]{S-17}).

\item
In view if the stabilization results for Betti numbers as in Theorem \ref{thm:stabilization}
it is tempting to hope for a similar statement for the
Castelnuovo-Mumford regularity $\reg M_S$ of $M_S$.
In particular, see \cite[Theorem A]{CE} for a related result where
$\Ab=\XI{0}$. But an answer to this question in general must be more complicated.
See Example \ref{ex:Koszul-lin2} below and its discussion which shows
that stabilization in a strict sense can not be expected.
\end{enumerate}
\end{rem}

\begin{cor}
Let $\Ab= (\XO{1})^{\otimes c}$ or $\Ab=(\XI{1})^{\otimes c}$.
Then the conclusions of Theorems \ref{thm:stabilization-syz} and \ref{thm:stabilization} hold.
\end{cor}
\begin{proof}
Apply Theorem \ref{thm:fg-gives-noeth-mod} and Theorems \ref{thm:stabilization-syz} and \ref{thm:stabilization}.
\end{proof}

As pointed out in \cite{NR} every polynomial ideal in finitely many variables gives rise to various $\FIO$- and $\FI$-modules.

\begin{ex}
\
\begin{enumerate}
\item
To illustrate this point, consider non-zero polynomials $f_5, f_6, f_7, f_8$ in $K[x_1, x_2, x_3]$ of degrees 5, 6, 7, 8, resp. Let $J$ be the ideal of $K[x_1,x_2,\ldots]$ that is generated by the $\Sym (\infty)$-orbits of $f_5, f_6, f_7, f_8$. It determines an $\FI$-ideal $I$ of $\XI{1}$, where $I_n = J \cap K[x_1,\ldots,x_n]$. Hence $I_n$ has a generating set of $4 \cdot n (n-1)(n-2)$ polynomials. Theorems \ref{thm:stabilization-syz} and \ref{thm:stabilization} say that, for any integer $p \ge 0$, there are finitely many master syzygies that determine all $p$-th syzygies of $I_n$ over $K[x_1,\ldots,x_n]$ for $n \gg 0$. Moreover, by \cite[Theorem 7.10]{NR}, the dimension of $K[x_1,\ldots,x_n]/I_n$ is eventually a linear function in $n$, and $\deg I_n$ is eventually an exponential function in $n$. In particular, $\lim_{n \to \infty} \sqrt[n]{\deg I_n}$ exists and is a positive rational number. It would be interesting to determine these asymptotic invariants.
\item
Similar questions arise when one varies Example \ref{exa:det} by considering determinantal ideals $I_n$ that are generated by the $t \times t$ minors of a generic $c \times n$ matrix $X_n$ that involve only the columns of the $\Inc$-orbits of a given $t$-tuple of distinct positive integers. For example, taking $t = 3$ and using the $\Inc$-orbit of $(1,5,7)$, the ideal $I_n$ is generated by $3$-minors whose column indices are in the set $\{(j_1, j_2, j_3) \; | \; 1 \le j_1,\ j_1 + 4 \le j_2,\ 2 + j_2 \le j_3 \le n\}$. The ideals $I_n$ determine an $\FIO$-ideal of $(\XO{1})^{\otimes c}$. Thus, their $p$-th syzygies stabilize in the above sense and $\dim K[x_1,\ldots,x_n]/I_n$ and $\deg I_n$ grow eventually linearly and exponentially, respectively.
\end{enumerate}
\end{ex}


\section{$\FIO$-Koszul complexes}
\label{sec:koszul}

There is a canonical complex of free $\FIO$-modules over any $\FIO$-algebra $\Ab$ that is analogous to the classical Koszul complex. It gives examples of projective resolutions as used in the previous section.

We need some notation.
Choose any $a \in A_O$, where $O = \{o\}$ is a one-element set. For every element $t$ of a totally ordered set $T \neq \emptyset$, let $\nu_{T, t}\colon O \to T$ be the map with $\nu_{T, t} (o) = t$. Set
\[
a_{T, t} = \Ab (\nu_{T, t}) (a) = \nu_{T, t}^* (a).
\]
Note that every $a_{T, t}$ is in $\Ab_T$.

\begin{rem}
\
\begin{enumerate}
\item
For every $\eps \in \Hom_{\FIO} (S, T)$, the definitions imply $\eps \circ \nu_{S, s} = \nu_{T, \eps (s)}$, which gives, for every
$s \in |S|$,
\begin{align}
\label{eq:compatible}
\eps^* (a_{S, s}) = a_{T, \eps (s)}.
\end{align}
\item
Let $\Ib = \langle a \rangle_{\Ab}$ be the ideal of $\Ab$ that is generated by $a$. Then, for every totally ordered set $T$, the elements $a_{T, 1},\ldots,a_{T, |T|}$ generate the ideal $\Ib_T$ of $\Ab_T$.
\end{enumerate}
\end{rem}

We use the above notation to define certain $\FIO$-morphisms.

\begin{defn}
For any integer $d > 0$, define a natural transformation $\phi_{d}\colon \Fo{d} \to \Fo{d-1}$ of free modules over $\Ab$ by
\[
\phi_d (S) (b e_{\pi}) = b \cdot \sum_{j = 1}^d (-1)^{j+1} a_{S, \pi (j)} e_{\pi_j},
\]
where $b \in \Ab_S$, $\pi \in \Hom_{\FIO} ([d], S)$, $\eps \in \Hom_{\FIO} (S, T)$, and $\pi_j\colon [d] \setminus \{j\} \to S$ is the restriction of $\pi$.
\end{defn}

Using Equation \eqref{eq:compatible}, one checks that this gives indeed a natural transformation. It is determined by $a$. Furthermore, these morphisms can be used to form an infinite complex.

\begin{lem}
\label{lem:Koszul-complex}
Let $\Ab$ be an $\FIO$-algebra. Every element $a \in \Ab_1$ determines a complex of free $\FIO$-modules over $\Ab$
\begin{align}
\label{eq:Koszul-complex}
\cdots \longrightarrow \Fo{d} \stackrel{\phi_{d}}{\longrightarrow} \Fo{d-1} \longrightarrow \cdots \longrightarrow \Fo{1} \stackrel{\phi_{1}}{\longrightarrow} \Fo{0} = \Ab \longrightarrow 0.
\end{align}
\end{lem}

\begin{proof}
For every totally ordered finite set $T$, the component of the Complex \eqref{eq:Koszul-complex} at $T$
\begin{align*}
0  \longrightarrow \Fo{|T|}_T \longrightarrow \Fo{|T|-1}_T \longrightarrow \cdots \longrightarrow \Fo{1}_T \longrightarrow\Fo{0}_T = \Ab_T \longrightarrow 0
\end{align*}
has length $|T|$ because $\Fo{d}_T = 0$ if $d > |T|$. (Note that this also shows that $\Fo{d}$ is not isomorphic to the $d$-th exterior power of $\Fo{1}$.) In fact, this component is the classical Koszul complex on the elements $a_{T,1},\ldots,a_{T, |T|}$ with coefficients in $\Ab_T$. Thus, the claim follows.
\end{proof}

We call the complex in Lemma~\ref{lem:Koszul-complex} the \emph{Koszul complex} on $a$ with coefficients in $\Ab$.
Using the classical characterization of the exactness of a Koszul complex, we obtain an analogous result for $\FIO$-algebras.

\begin{prop}
\label{prop:Koszul-exact}
Let $a \in \Ab_1$ be an element of an $\FIO$-algebra $\Ab$. If, for every totally ordered finite set $T$, the elements $a_{T,1},\ldots,a_{T, |T|}$ form an $\Ab_T$-regular sequence, then the Koszul complex on $a$ with coefficients in $\Ab$ is acyclic.
\end{prop}

\begin{ex}
Consider the ideal $\Ib$ of $\XO{1}$ that is generated by $x_1$. Then $\Ib_n = \langle x_1,\ldots,x_n \rangle$ is generated by a regular sequence for every $n \in \N$, and so the Koszul complex determined by $x_1$ is acyclic.
\end{ex}

\begin{rem}
If $\Ab$ is a $\Z$-graded $\FIO$-algebra and $a$ is homogeneous, then the Koszul complex on $a$ is complex of graded $\FIO$-modules if one uses suitable degree shifts. For an integer $k$ and a graded $\FIO$-module $\M$, we denote by $\M (k)$ the module with the same module structure as $\M$, but with an (internal) grading given by
$[((\M (k))_n]_j = [\M_n]_{j + k}$ for all $j \in \Z$. We illustrate this by an example.
\end{rem}

\begin{ex}
\label{ex:Koszul-lin2}
Let $\Ib$ be the ideal of $\XO{1}$ that is generated by $x_1^k$ for some $k \in \N$. Then $\Ib_n = \langle x_1^k,\ldots,x_n^k \rangle$ for every $n \in \N$. Thus, the graded Koszul complex on $x_1^k$ is
\begin{align*}
\cdots \to \Fo{d}(-d k) \to\Fo{d-1}(- (d-1)k) \to \cdots \to \Fo{1}(-k) \to \XO{1} \to 0.
\end{align*}
For every fixed integer $p$, the degrees of generators of the $p$-th syzygy module of $\Ib_n$ are bounded above by a constant that is independent of $n$, which is in line with the above stabilization result. In fact, the $p$-syzygies are generated in degree $p k$.

In contrast to the situation for $\FI$-modules with constant coefficients, that is, over $\XI{0}$ (see \cite{CE}), this shows that for modules over $\XO{1}$ the degrees of the generators of the $p$-th syzygy modules cannot be uniformly bounded above independent of $p$. Furthermore, observe that the number of minimal generators of $I_n$ grows with $n$. Thus, the recent boundedness results in \cite{Draisma-noeth} (see also \cite{ESS}) do not apply directly to the categories $\FIO$-$\Mod (\Ab)$ and $\FI$-$\Mod (\Ab)$.
\end{ex}


\end{document}